\documentclass[a4paper,10pt]{amsart}
\usepackage[utf8x]{inputenc}
\usepackage{amsfonts}
\usepackage{amsmath}
\usepackage{amssymb}
\usepackage{graphicx}
\usepackage[all]{xy}
\usepackage{mathtools}

\usepackage[draft=false]{hyperref}  

\DeclareMathOperator{\id}{id}
\DeclareMathOperator{\Hom}{Hom}
\DeclareMathOperator{\Mor}{Mor}

\newcommand{\OFG}{\mathcal{O}_\mathfrak{F}G}
\newcommand{\uR}{\underline{R}}
\newcommand{\uZZ}{\ensuremath{\underline{\mathbb{Z}}}}
\newcommand{\uE}{\underline{\mathrm{E}}}
\newcommand{\B}{\mathrm{B}} 

\DeclareMathOperator{\ucd}{\underline{cd}}
\DeclareMathOperator{\uFP}{\underline{FP}}
\DeclareMathOperator{\ugd}{\underline{gd}}
\DeclareMathOperator{\vcd}{vcd}
\DeclareMathOperator{\cd}{cd}
\DeclareMathOperator{\FP}{FP}

\DeclareMathOperator{\Tor}{Tor}
\newcommand*{\RR}{\ensuremath{\mathbb{R}}}

\newcommand*{\FF}{\ensuremath{\mathbb{F}}}
\newcommand*{\ZZ}{\ensuremath{\mathbb{Z}}}

\newcommand*{\QQ}{\ensuremath{\mathbb{Q}}}
\newcommand*{\DD}{\ensuremath{\mathbb{D}}}
\DeclareMathOperator{\Ind}{Ind}
\DeclareMathOperator{\Res}{Res}

\newcommand*{\longhookrightarrow}{\ensuremath{\lhook\joinrel\relbar\joinrel\rightarrow}}

\newcommand*{\longtwoheadrightarrow}{\ensuremath{\relbar\joinrel\twoheadrightarrow}}

\DeclareMathOperator{\Fcd}{\mathfrak{F}cd}
\DeclareMathOperator{\Gcd}{Gcd}

\usepackage{amsthm}
\newtheorem{Lemma}{Lemma}[section]
\newtheorem{Theorem}[Lemma]{Theorem}
\newtheorem{Cor}[Lemma]{Corollary}
\newtheorem{Prop}[Lemma]{Proposition}
\theoremstyle{definition}

\newtheorem{Example}[Lemma]{Example}
\newtheorem{Question}[Lemma]{Question}
\newtheorem*{Acknowledgements}{Acknowledgements}
\theoremstyle{remark}
\newtheorem{Remark}[Lemma]{Remark}

\newcommand\mf\mathfrak

\title{Bredon--Poincar\'e Duality Groups}
\author{Simon St. John-Green}
\email{Simon.StJG@gmail.com}
\address{Department of Mathematics, University of Southampton, SO17 1BJ, UK}
\date{\today}


\begin{document}

\numberwithin{equation}{section}

{\abstract{If $G$ is a group which admits a manifold model for $\B G$ then $G$ is a Poincar\'e duality group.  We study a generalisation of Poincar\'e duality groups, introduced initially by Davis and Leary in \cite{DavisLeary-DiscreteGroupActionsOnAsphericalManifolds}, motivated by groups $G$ with cocompact manifold models $M$ for $\uE G$ where $M^H$ is a contractible submanifold for all finite subgroups $H$ of $G$.  We give several sources of examples and constructions of these Bredon--Poincar\'e duality groups, including using the equivariant reflection group trick of Davis and Leary to construct examples of Bredon--Poincar\'e duality groups arising from actions on manifolds $M$ where the dimensions of the submanifolds $M^H$ are specified.  We classify Bredon--Poincar\'e duality groups in low dimensions, and discuss behaviour under group extensions and graphs of groups.}}
\maketitle

\section{Introduction}

A \emph{duality group} is a group $G$ of type $\FP$ for which 
\[H^i(G, \ZZ G) \cong \left\{ \begin{array}{l l} \text{$\ZZ$-flat} & \text{if $i = n$} \\ 0 & \text{else.} \end{array}\right.\]
Where $n$ is necessarily the cohomological dimension of $G$.  The name duality comes from the fact that this condition is equivalent to existence of a $\ZZ G$ module $D$, giving an isomorphism
\begin{equation*} H^i(G, M) \cong H_{n-i}(G, D \otimes_\ZZ M) \tag{$\ast$} \end{equation*}
for all $i$ and all $\ZZ G$-modules $M$.  It can be proven that given such an isomorphism, the module $D$ is necessarily $H^n(G, \ZZ G)$.  A duality group $G$ is called a \emph{Poincar\'e duality group} if in addition 
\[H^i(G, \ZZ G) \cong \left\{ \begin{array}{l l} \ZZ & \text{if $i = n$} \\ 0 & \text{else.} \end{array}\right.\]
These groups were first defined by Bieri \cite{Bieri-GruppenMitPoincareDualitat}, and independently by Johnson--Wall \cite{JohnsonWall-OnGroupsSatisfyingPoincareDuality}.  Duality groups were first studied by Bieri and Eckmann in \cite{BieriEckmann-HomologicalDualityGeneralizingPoincareDuality}.  See \cite{Davis-PoincareDualityGroups} and \cite[III]{Bieri-HomDimOfDiscreteGroups} for an introduction to these groups.  

If a group $G$ has a manifold model for $\B G$ then $G$ is a Poincar\'e duality group.  Wall asked if the converse is true \cite{Wall-HomologicalGroupTheory}---the answer is no as Poincar\'e duality groups can be built which are not finitely presented \cite[Theorem C]{Davis-CohomologyCoxeterGroupRingCoeff}---but the question remains a significant open problem if we include the requirement that $G$ be finitely presented.  The conjecture is known to hold only in dimension $2$ \cite{Eckmann-PoincareDualityGroupsOfDimensionTwoAreSurfaceGroups}.

Let $R$ be a commutative ring.  A group $G$ is \emph{duality over $R$} if $G$ is $\FP$ over $R$ and
\[H^i(G, R G) \cong \left\{ \begin{array}{l l} \text{$R$-flat} & \text{if $i = n$} \\ 0 & \text{else.} \end{array}\right.\]
$G$ is \emph{Poincar\'e duality over $R$} if
\[ H^i(G, R G) \cong \left\{ \begin{array}{l l} R & \text{if $i = n$} \\ 0 & \text{else.} \end{array}\right.\]
An analog of Wall's conjecture is whether every torsion-free finitely presented Poincar\'e duality group over $R$ is the fundamental group of an aspherical closed $R$-homology manifold \cite[Question 3.5]{Davis-PoincareDualityGroups}.  This is answered in the negative by Fowler for $R = \QQ$ \cite{Fowler-FinitenessForRationalPoincareDuality}, but remains open for $R = \ZZ$.

We study a generalisation of Poincar\'e duality groups, looking at the algebraic analog of the condition that $G$ admit a manifold model $M$ for $\uE G$ such that for any finite subgroup $H$ the fixed point set $M^H$ is a submanifold.  Here $\uE G$ refers to the classifying space for proper actions of $G$, this is a $G$-CW complex $X$ such that for all subgroups $H$ of $G$, $X^H$ is contractible if $H$ is finite and empty otherwise.  Such spaces are unique up to homotopy equivalence and we denote the minimal dimension of a classifying space for proper actions by $\ugd G$.  The cohomology theory most suited to the study of proper actions is Bredon cohomology.  For instance, writing $\ucd G$ for the Bredon cohomological dimension, L\"uck and Meintrup have shown that $\ugd G = \ucd G$ except for the possibility that $\ucd G = 2$ and $\ugd G = 3$ which Brady, Leary and Nucinkis have shown can occur \cite{LuckMeintrup-UniversalSpaceGrpActionsCompactIsotropy}\cite{BradyLearyNucinkis-AlgAndGeoDimGroupsWithTorsion}.  The Bredon cohomology analog of the $\FP$ condition will be denoted $\uFP$, and that of $\cd G$ will be denoted $\ucd G$.

If $G$ admits a cocompact manifold model $M$ for $\uE G$ then $G$ is $\uFP$.  Also if for any finite subgroup $H$ the fixed point set $M^H$ is a submanifold, we have the following condition on the cohomology of the Weyl groups $WH = N_GH/H$:
\[ H^i(WH, \ZZ [WH]) =  \left\{ \begin{array}{c c} \ZZ & \text{ if } i = \dim M^H  \\ 0 & \text{ else. } \end{array} \right.  \]
See \cite[p.3]{DavisLeary-DiscreteGroupActionsOnAsphericalManifolds} for a proof of the above.  
Building on this, in \cite{DavisLeary-DiscreteGroupActionsOnAsphericalManifolds} and also in \cite[Definition 5.1]{MartinezPerez-EulerClassesAndBredonForRestrictedFamilies} a \emph{Bredon duality group over $R$} is defined as a group $G$ of type $\uFP$ such that for every finite subgroup $H$ of $G$ there is an integer $n_H$ with 
\[ H^i(WH, R [WH]) = \left\{ \begin{array}{l l} \text{$R$-flat} & \text{if $i = n_H$} \\ 0 & \text{else.} \end{array} \right. \]
Furthermore, $G$ is said to be Bredon--Poincar\'e duality over $R$ if for all finite subgroups $H$, 
\[H^{n_H}(WH, R [WH]) = R\]
We say that a Bredon duality group $G$ is \emph{dimension $n$} if $\ucd G = n$.  Note that for torsion-free groups these reduce to the usual definitions of duality and Poincar\'e duality groups.  

One might generalise Wall's conjecture: Let $G$ be Bredon--Poincar\'e duality over $\ZZ$, such that $WH$ is finitely presented for all finite subgroups $H$.  Does $G$ admit a cocompact manifold model $M$ for $\uE G$, where for each finite subgroup $H$ the fixed point set $M^H$ is a submanifold?  This is false by an example of Jonathon Block and Schmuel Weinberger, suggested to us by Jim Davis.

\theoremstyle{plain}\newtheorem*{CustomThmBPDNoMfd}{Theorem \ref{theorem:BPD groups with no manifold model}}
\begin{CustomThmBPDNoMfd}
  There exist examples of Bredon--Poincar\'e duality groups over $\ZZ$, such that $WH$ is finitely presented for all finite subgroups $H$ but $G$ doesn't admit a cocompact manifold model $M$ for $\uE G$.
\end{CustomThmBPDNoMfd}

If $G$ is Bredon--Poincar\'e duality and virtually torsion-free then $G$ is virtually Poincar\'e duality.  Thus an obvious question is whether all virtually Poincar\'e duality groups are Bredon--Poincar\'e duality, in \cite{DavisLeary-DiscreteGroupActionsOnAsphericalManifolds} it is shown that this is not the case for $R = \ZZ$.  An example is also given in \cite[\S 6]{MartinezPerez-EulerClassesAndBredonForRestrictedFamilies} which fails for both $R = \ZZ$ and for $R = \FF_p$, the finite field of $p$ elements.  One might also ask if every Bredon--Poincar\'e duality group is virtually torsion-free but this is also not the case, see for instance Examples \ref{example:FarbWeinberger} and \ref{example:disc sub of lie not vtf}. 

In \cite[Theorems D,E]{Hamilton-WhenIsGroupCohomologyFinitary} Hamilton shows that, over a field $F$ of characteristic $p$, given an extension $\Gamma$ of a torsion-free group of type $\FP_\infty$ by a finite $p$-group, the resulting group will be of type $\uFP_\infty$ (by examples of Leary and Nucinkis, an extension by an arbitrary finite group may not even be $\uFP_0$ \cite{LearyNucinkis-SomeGroupsOfTypeVF}).  Martinez-Perez builds on this result to show that if $G$ is assumed Poincar\'e duality then $\Gamma$ is Bredon--Poincar\'e duality over $F$ with $\ucd_F \Gamma  =\cd_F G$.  However her results do not extend to Bredon duality groups.

Given a Bredon duality group $G$ we write $\mathcal{V}(G)$ for the set
\[\mathcal{V}(G) = \{ n_F : F \text{ a non-trivial finite subgroup of }G\} \subseteq \{0, \ldots, n \}\]
In Example \ref{example:duality arbitrary V} we will build Bredon duality groups with arbitrary $\mathcal{V}(G)$.  If $G$ has a manifold model, or homology manifold model, for $\uE G$ then there are some restrictions on $\mathcal{V}(G)$---see Section \ref{subsection:actions on homology manifolds} for this as well as for the definition of homology manifold.  In Section \ref{section:reflection groups} we build Bredon--Poincar\'e duality groups with a large class of vectors $\mathcal{V}(G)$, however the following question remains open:

\begin{Question}\label{question:duality prescribed VG}
 Is it possible to construct Bredon--Poincar\'e duality groups with prescribed $\mathcal{V}(G)$?
\end{Question}

One can show that for a Bredon--Poincar\'e duality group, $n_{1} \le \ucd G$ (recall $n_1$ is the integer for which $H^n(G, RG) \cong R$) and also, if we are working over $\ZZ$, then $n_{1} = \cd_\QQ G$ (Lemma \ref{Lemma:observations}).  Thus the following question is of interest:

\begin{Question}\label{question:duality n=nid}
 Do there exist Bredon duality groups with $\ucd G \neq n_{1}$?
\end{Question}

Examples of groups for which $\cd_\QQ G \neq \ucd_\ZZ G$ are known \cite{LearyNucinkis-SomeGroupsOfTypeVF}, but there are no known examples of type $\uFP_\infty$.  This question is also related to \cite[Question 5.8]{MartinezPerez-EulerClassesAndBredonForRestrictedFamilies} where it is asked whether a virtually torsion-free Bredon duality group satisfies $\ucd G = \vcd G$.

One might hope to give a definition of Bredon--Poincar\'e duality groups in terms of Bredon cohomology only, we do not know if this is possible but we show in Section \ref{section:wrong notion of duality} that the na\"ive idea of asking that a group be $\uFP$ with 
\[ H^i_{\mathfrak{F}}(G, R[?,-]) \cong \left\{ \begin{array}{l l} \uR & \text{if $i = n$} \\ 0 & \text{else.} \end{array} \right.\]
is not the correct definition, where in the above $H^i_{\mathfrak{F}}$ denotes the Bredon cohomology and $\uR$ is the constant covariant Bredon module (these will be defined in more detail in Section \ref{section:finiteness conditions in Bredon}).  Namely we show in Theorem \ref{theorem:bredon wrong duality} that any such group is necessarily a torsion-free Poincar\'e duality group over $R$.

\begin{Acknowledgements}
 The author would like to thank his supervisor Brita Nucinkis for suggesting the topic and for many enlightening mathematical conversations.  Example \ref{example:DicksLeary example PD over R not Z} and the examples in Section \ref{section:reflection groups} are due to Ian Leary and the author would like to thank him for suggesting them and for other helpful discussions.  The author would like to thank Jim Davis for showing us that some groups constructed by Block and Weinberger, appearing now in Section \ref{section:counterexample to generalised PDn}, answered a question in an earlier version of this article.
\end{Acknowledgements}

\section{A Review of Finiteness Conditions in Bredon Cohomology}\label{section:finiteness conditions in Bredon}
This section contains a review of Bredon cohomology and finiteness conditions in Bredon cohomology.

Fix a group $G$ and commutative ring $R$, and let $\mathfrak{F}$ denote the family of all finite subgroups of $G$.  The \emph{orbit category}, denoted $\OFG$, is the small category whose objects are the transitive $G$-sets $G/H$ for $H \in \mf{F}$ and whose arrows are all $G$-maps between them.  Any $G$-map $G/H \to G/K$ is determined entirely by the image of the coset $H$ in $G/K$, and $H \mapsto xK$ determines a $G$-map if and only if $ x^{-1}Hx \le K$.

A \emph{Bredon module} is a contravariant functor from $\OFG$ to the category of left $R$-modules.  As such the category of Bredon modules is Abelian and exactness is defined pointwise---a short exact sequence 
\[ M^\prime \longrightarrow M \longrightarrow M^{\prime\prime} \]
is exact if and only if 
\[ M^\prime(G/H) \longrightarrow M(G/H) \longrightarrow M^{\prime\prime}(G/H) \]
is exact for all $H \in \mf{F}$.  If $\Omega_1$ and $\Omega_2$ are $G$-sets then we denote by $\ZZ[\Omega_1, \Omega_2]$ the free Abelian group on the set of all $G$-maps $\Omega_1 \to \Omega_2$. If $K \in \mf{F}$, the Bredon module $R[-,G/K]$ defined by 
\[ R[-,G/K](G/H) = R \otimes \ZZ[G/H,G/K] \]
is free, and taking direct sums of these gives all free Bredon modules.  Using these one can show the category of Bredon modules has enough projectives. 

Although we will rarely need them, in fact only in Section \ref{section:wrong notion of duality}, one can also define \emph{covariant Bredon modules} as covariant functors from $\OFG$ to the category of left $R$-modules.  As in the contravariant case the category of covariant Bredon modules is Abelian and exactness is defined pointwise, but the free covariant Bredon modules are direct sums of the module is $R[G/K, -]$, defined by
\[ R[G/K,-](G/H) = R \otimes \ZZ[G/K,G/H]\] 

If $M$ and $N$ are Bredon modules then we let $\Mor_{\mf{F}}(M, N)$ denote the $R$-module of Bredon module homomorphisms between $M$ and $N$---the natural transformations from $M$ to $N$.  Similarly one can define the group of Bredon module homomorphisms between covariant Bredon modules.  Given a Bredon module $M$ and a covariant Bredon module $A$ the tensor product, denoted $M \otimes_{\mathfrak{F}} A$, is
\[ M \otimes_{\mathfrak{F}} A = \left. \bigoplus_{G/H \in \OFG} M(G/H) \otimes_R A(G/H) \right/ \sim \]
Where $\alpha^*(m) \otimes a \sim m \otimes \alpha_*(a)$ for all morphisms $\alpha: G/H \to G/K$ in $\OFG$, elements $m \in M(G/K)$ and $a \in A(G/H)$, and $G/H,G/K \in \OFG$.  

\begin{Lemma}[The Yoneda-type Lemma]\label{lemma:yoneda-type}\cite[p.9]{MislinValette-BaumConnes}
For any Bredon module $M$ and $G/H \in \OFG$ there is an isomorphism, natural in $M$:
 \begin{align*}
\Mor_{\mathfrak{F}} \left( R[-,G/H], M\right) &\cong M(G/H) \\
 f &\mapsto f(G/H)(\id_{G/H})
 \end{align*}
Similarly for any covariant Bredon module $A$:
\[\Mor_{\mathfrak{F}} \left( R[G/H,-], A\right) \cong A(G/H)\]
\end{Lemma}

A Bredon module $M$ is said to be finitely generated if it admits an epimorphism 
\[\bigoplus_{i \in I} R[-,G/H_i] \longtwoheadrightarrow M\]
with $I$ a finite set, is said to be $\uFP_n$ if it admits a projective resolution which is finitely generated in all degrees $\le n$ and is said to be $\uFP_\infty$ if it is $\uFP_n$ for all $n$.  The Bredon cohomological dimension of a Bredon module $M$ is the shortest length of a projective resolution of $M$.

We denote by $\uR$ the constant Bredon module on $R$, sending $G/H$ to $R$ for all finite $H$ and sending all morphisms to the identity.  A group $G$ is $\uFP_n$ if $\uR$ is $\uFP_n$ and the cohomological dimension of $G$, denoted $\ucd G$, is the shortest length of a projective resolution of $\uR$.  A group is $\uFP$ if it is $\uFP_\infty$ and has $\ucd G < \infty$.  If we need to emphasize the ring $R$ then we write $\ucd_R$ and ``$\uFP_n$ over $R$''.

The Bredon cohomology of a group $G$ with coefficients in a Bredon module $M$ is defined to be 
\[H^*_{\mf{F}} (G, M) \cong H^* \Mor_{\mf{F}}(P_*, M)\]
Where $P_*$ is a projective resolution of $\uR$.  The Bredon cohomological dimension can then be restated as
\[\ucd G = \sup\{ n : H^n_{\mf{F}}(G, M) \neq 0 \text{ for some Bredon module $M$. }\} \]
If $A$ is a covariant Bredon module then the Bredon homology of $G$ with coefficients in $A$ is defined as
\[H_*^{\mf{F}} (G, A) \cong H_* \left( P_* \otimes_{\mf{F}} A \right)\]
One can define flat modules and the Bredon homological dimension, but we will not require these.

\begin{Lemma}\label{lemma:OFFPn characterisation}\cite[Lemma 3.1, Lemma 3.2]{KMN-CohomologicalFinitenessConditionsForElementaryAmenable}, 
\begin{enumerate}
 \item $G$ is $\uFP_0$ if and only if $G$ has finitely many conjugacy classes of finite subgroups.
 \item A Bredon module $M$ is $\uFP_n$ ($n \ge 1$) if and only if $G$ is $\uFP_0$ and $M(G/K)$ is of type $\FP_n$ over the Weyl group $WK = N_GK/ K$ for all finite subgroups $K \le G$.
 \item $G$ is $\uFP_n$ if and only if the Weyl groups $WK = N_GK/K$ are $\FP_n$ for all finite subgroups $K$.
\end{enumerate}
\end{Lemma}
Note that asking for $WK$ to be $\FP_n$ is equivalent to asking that the normalisers $N_GK$ are $\FP_n$ or that the centralisers $C_GK$ are $\FP_n$.

Many results about finiteness in ordinary group cohomology carry over into the Bredon case, for example in in \cite[Section 5]{MartinezNucinkis-GeneralizedThompsonGroups}, a version of the Bieri-Eckmann criterion is proven (see \cite[Section 1.3]{Bieri-HomDimOfDiscreteGroups} for the classical case).

Finally, we need the following two easy lemmas.

\begin{Lemma}\label{lemma:uFPn over Z then over R}
 If $G$ is $\uFP_n$ over $\ZZ$ then $G$ is $\uFP_n$ over $R$.
\end{Lemma}
\begin{proof}
 Let $P_*$ be a projective resolution of $\uZZ$, then replacing each module $P_i$ with the module 
\[P_i^\prime : G/H \mapsto R \otimes P_i(G/H)\]
gives a projective resolution of $\uR$ by projective modules.  The resolution remains exact since $P_i(G/H)$ is a $\ZZ$-split resolution for all finite subgroups $H$.
\end{proof}

\begin{Lemma}\label{lemma:Rtf then cdRG le ucdRG}
 If $G$ is $R$-torsion-free then $\cd_R G \le \ucd_R G$.
\end{Lemma}
\begin{proof}
 For any finite subgroup $H$ of $G$, evaluating $R[-,G/H]$ at $G/1$ gives $R[G/H]$ and the natural projection $RG \longrightarrow R[G/H]$ is split by the map 
\[ H \longmapsto \frac{1}{\lvert H \rvert} \sum_{h \in H} h \]
 Hence evaluating any projective Bredon module at $G/1$ gives a projective $RG$-module.  The result follows by evaluating a length $n$ projective resolution of $\uR$ at $G/1$, where $n = \ucd_R G$.
\end{proof}

\section{Preliminary Observations}

A group $G$ is \emph{$R$-torsion-free} if the order of every finite subgroup of $G$ is invertible in $R$.  Equivalently one can show using Cauchy's theorem \cite[1.6.17]{Robinson} that this is equivalent to the order of every finite order element being invertible in $R$.

Recall that a Bredon duality group is said to be dimension $n$ if $\ucd G = n$.

\begin{Lemma}\label{Lemma:observations}~
\begin{enumerate}
\item If $G$ is Bredon duality of dimension $n$ over $\ZZ$ then $n_H = \cd_\QQ WH$ for all finite $H$, and $n_{1} \le n$.
\item If $G$ is $R$-torsion-free and Bredon duality of dimension $n$ over $R$ then $n_H = \cd_R WH$ and $n_{1} \le n$.
 \end{enumerate}
\end{Lemma}

To prove the Lemma we need the following proposition, an analog of \cite[VIII.6.7]{Brown} for arbitrary rings $R$ and proved in exactly the same way.
\begin{Prop}\label{prop:FP over Q then cdQ = max H^n}
 If $G$ is $\FP$ over $R$ then $\cd_R G = \max \{ n  :  H^n(G, R  G) \neq 0\}$.
\end{Prop}

\begin{proof}[Proof of Lemma \ref{Lemma:observations}]
\begin{enumerate}
 \item Since $G$ is $\uFP$, $WH$ is $\FP_\infty$ for all finite $H$ (Lemma \ref{lemma:OFFPn characterisation}) and we may apply \cite[Corollary 3.6]{Bieri-HomDimOfDiscreteGroups} to get a short exact sequence
\[0 \to H^q(WH, \ZZ [WH]) \otimes_\ZZ \QQ \to H^q(WH, \QQ \otimes_\ZZ \ZZ [WH]) \]
\[\to \Tor_1^\ZZ(H^{q+1}(WH, \ZZ [WH]) ,\QQ) \to 0\]
$H^{q+1}(WH, \ZZ [WH])$ is $\ZZ$-flat for all $q$ giving an isomorphism
\[H^q(WH, \ZZ [WH]) \otimes_\ZZ \QQ \cong H^q(WH, \QQ [WH])\]
Proposition \ref{prop:FP over Q then cdQ = max H^n} shows $n_H = \cd_\QQ WH$.  Finally, $\cd_\QQ G \le \ucd_\ZZ G$ for all groups $G$ by \cite[Theorem 2]{BradyLearyNucinkis-AlgAndGeoDimGroupsWithTorsion}, so $n_{1} \le n$.
\item If $G$ is $R$-torsion free then for any finite subgroup $H$, 
\[\cd_R N_GH \le \cd_R G \le \ucd_R G\]
and $N_GH$ is $\FP_\infty$ over $R$ by Lemma \ref{lemma:OFFPn characterisation}.  Since 
\[H^i(N_G H, R[N_G H]) \cong H^i(WH, R[WH])\]
Proposition \ref{prop:FP over Q then cdQ = max H^n} shows $n_H = \cd_R N_GH = \cd_R WH$.  Finally, $n_{1} \le n$ because $\cd_R G \le \ucd_R G$ (Lemma \ref{lemma:Rtf then cdRG le ucdRG}).
\end{enumerate} 
\end{proof}

\begin{Lemma}\label{lemma:uDn over Z then uDn over R}
 If $G$ is Bredon duality of dimension $n$ over $\ZZ$ then $G$ is Bredon duality of dimension $n$ over any ring $R$.
\end{Lemma}

\begin{proof}
 Since $G$ is $\uFP$ over $\ZZ$, $G$ is $\uFP$ over $R$ (Lemma \ref{lemma:uFPn over Z then over R}). As in the proof of part (1) of the previous lemma there is an isomorphism for any finite subgroup $H$,
\[H^q(WH, \ZZ [WH]) \otimes_\ZZ R \cong H^q(WH, R [WH])\]
Observing that if an Abelian group $M$ is $\ZZ$-flat then $M \otimes_\ZZ R$ is $R$-flat completes the proof.
\end{proof}

In the proposition below $\Fcd G$ denotes the $\mathfrak{F}$-cohomological dimension introduced in \cite{Nucinkis-CohomologyRelativeGSet} and $\Gcd G$ denotes the Gorenstein cohomological dimension, see for example \cite{BahlekehDembegiotiTalelli-GorensteinDimensionAndProperActions}.

\begin{Prop}
 If $G$ is a Bredon-duality group over $R$ then $\Gcd G = \Fcd G = n_1$ and if $G$ is virtually torsion-free then $\vcd G = n_1$ also.
\end{Prop}
\begin{proof}
 This proof uses an argument due to Degrijse and Mart\'{\i}nez--P\'erez in \cite{DegrijseMartinezPerez-DimensionInvariants}.  By \cite[Theorem 2.20]{Holm-GorensteinHomologicalDimensions} the Gorenstein cohomological dimenion, denoted $\Gcd G$, can be characterised as
 \[ \Gcd G = \sup \{ n : H^n(G, P) \neq 0 \text{ for $P$ any projective $RG$-module } \}  \] 
 As $G$ is $\FP_\infty$ we need only check when $P = RG$ and hence $\Gcd G = n_1$.  Since $\Fcd G \le \ucd G < \infty$, we can conclude that $\Fcd G = \Gcd G$ \cite[Theorem 3.11]{Me-GorensteinAndF} and finally for virtually torsion-free groups $\Fcd G = \vcd G$ \cite{MartinezPerezNucinkis-MackeyFunctorsForInfiniteGroups}.
\end{proof}

\section{Examples}
In this section we provide several sources of examples of Bredon duality and Bredon--Poincar\'e duality groups, showing that these properties are not too rare.

\subsection{Smooth Actions on Manifolds}\label{subsection:duality smooth actions on manifolds}

Recall from the introduction that if $G$ has a manifold model $M$ for $\uE G$ such that $M^H$ is a submanifold for all finite subgroups $H$ then $G$ is Bredon--Poincar\'e duality.  The following lemma guarantees that $M^H$ is a submanifold of $M$:
\begin{Lemma}\label{Lemma:proper and locally linear => fixed points are submanifolds}\cite[10.1 p.177]{Davis} If $G$ is a discrete group acting properly and locally linearly on a manifold $M$ then the fixed points subsets of finite subgroups of $G$ are submanifolds of $M$.
\end{Lemma}
Locally linear is a technical condition, the definition of which can be found in \cite[Definition 10.1.1]{Davis}, for our purposes it is enough to know that if $M$ is a smooth manifold and $G$ acts by diffeomorphisms then the action is locally linear.  The locally linear condition is necessary however---in \cite{DavisLeary-DiscreteGroupActionsOnAsphericalManifolds} examples are given of virtually torsion-free groups acting as a discrete cocompact group of isometries of a CAT(0) manifold which are not Bredon duality.

\begin{Example}Let $p$ be a prime and let $G$ be the wreath product
 \[G = \ZZ \wr C_p = \left( \bigoplus_{\ZZ_p} \ZZ \right) \rtimes C_p\]
Where $C_p$ denotes the cyclic group of order $p$.  $G$ acts properly and by diffeomorphisms on $\RR^p$: The copies of $\ZZ$ act by translation along the axes, and the $C_p$ permutes the axes.  The action is cocompact with fundamental domain the quotient of the $p$-torus by the action of $C_p$.  The finite subgroup $C_p$ is a representative of the only conjugacy class of finite subgroups in $G$, and has fixed point set the line $\{(\lambda,\cdots,\lambda)  :  \lambda \in \RR \}$.  
If $z = (z_1, \ldots, z_p) \in \ZZ^p$ then the fixed point set of $(C_p)^z$ is the line $\{ (\lambda + z_1, \ldots, \lambda + z_p)  :  \lambda \in \RR \}$.

Hence $\RR^p$ is a model for $\uE G$ and, invoking Lemma \ref{Lemma:proper and locally linear => fixed points are submanifolds},  $G$ is a Bredon--Poincar\'e duality group of dimension $p$ with $\mathcal{V} = \{1\}$.
\end{Example}

\begin{Example}\label{example:duality Zn antipodal}
Fixing positive integers $m \le n$, if $G = \ZZ^n \rtimes C_2$ where $C_2$, the cyclic group of order 2, acts as the antipodal map on $\ZZ^{n-m} \le \ZZ^n$ then 
\[ N_GC_2 = C_GC_2 = \{ g \in G  :  gz = zg \}  \]
But this is exactly the fixed points of the action of $C_2$ on $G$, hence $N_G C_2 = \ZZ^m \rtimes C_2$ and
\[ H^i(N_GC_2, R[N_G C_2]) \cong \left\{ \begin{array}{l l} R & \text{ if } i = m \\ 0 & \text{ else.} \end{array}\right. \]
$G$ embeds as a discrete subgroup of $\text{Isom}(\RR^n) = \RR^n \rtimes GL_n(\RR)$ and acts properly and cocompactly on $\RR^n$.  It follows that $G$ is $\uFP$ and $\ucd G = n$ so $G$ is Bredon--Poincar\'e duality of dimension $n$ over any ring $R$ with $\mathcal{V} = \{m\}$.
\end{Example}

\begin{Example}\label{example:duality Zn antipodal generalised}
Similarly to the previous example we can take 
\[ G = \ZZ^n \rtimes \bigoplus_{i=1}^n C_2 \]
Where the $j^\text{th}$ copy of $C_2$ acts antipodally on the $j^\text{th}$ copy of $\ZZ$ in $\ZZ^n$.  Note that $G$ is isomorphic to $(D_\infty)^n$ where $D_\infty$ denotes the infinite dihedral group.  As before $G$ embeds as a discrete subgroup of $\text{Isom}(\RR^n) = \RR^n \rtimes GL_n(\RR)$ and acts properly and cocompactly on $\RR^n$.  Thus $G$ is $\uFP$ and $\ucd G = n$, so $G$ is Bredon--Poincar\'e duality of dimension $n$ over any ring $R$ with $\mathcal{V}(G) = \{0, \ldots, n\}$.

More generally, we could take a subgroup $\bigoplus_{i=1}^m C_2 \longhookrightarrow \bigoplus_{i=1}^nC_2$ and form the semi-direct product of $\ZZ^n$ with this subgroup.  Although this gives us a range of possible values for $\mathcal{V}(G)$ it is impossible to produce a full range of values with this method.  For example one can show that a Bredon--Poincar\'e duality group of dimension $4$ with the form
\[G = \ZZ^4 \rtimes \bigoplus_{i = 1}^m C_2 \]
cannot have $\mathcal{V}(G) = \{1, 3\}$.  
\end{Example}

\begin{Example}\label{example:FarbWeinberger}In \cite[Theorem 6.1]{FarbWeinberger-IsometriesRigidityAndUniversalCovers}, Farb and Weinberger construct a group $G$ acting properly cocompactly and by diffeomorphisms on $\RR^n$ for some $n$.  Thus $G$ is a Bredon--Poincar\'e duality group, however it is not virtually torsion-free. 
\end{Example}

\subsection{A counterexample to the generalised \texorpdfstring{PD$^\text{n}$}{PDn} conjecture}\label{section:counterexample to generalised PDn}

Let $G$ be Bredon--Poincar\'e duality over $\ZZ$, such that $WH$ is finitely presented for all finite subgroups $H$.  One might ask if $G$ admits a cocompact manifold model $M$ for $\uE G$, where for each finite subgroup $H$ the fixed point set $M^H$ is a submanifold?  This is generalisation of the famous PD$^\text{n}$-conjecture, due to Wall \cite{Wall-HomologicalGroupTheory}.  This example is due to Jonathon Block and Schmuel Weinberger and was suggested to us by Jim Davis.

\begin{Theorem}\label{theorem:BPD groups with no manifold model}
 There exist examples of Bredon--Poincar\'e duality groups over $\ZZ$, such that $WH$ is finitely presented for all finite subgroups $H$ but $G$ doesn't admit a cocompact manifold model $M$ for $\uE G$.
\end{Theorem}

Combining Theorems 1.5 and 1.8 of \cite{BlockWeinberger-GeneralizedNielsenRealizationProblem} gives the following example.

\begin{Theorem}[Block--Weinberger]\label{theorem:block--weinberger}
There exists a short exact sequence of groups 
\[ 
1 \longrightarrow H \longrightarrow G \longrightarrow Q \longrightarrow 1
\]
with $Q$ finite, such that
\begin{enumerate}
 \item All torsion in $G$ is contained in $H$.
 \item There exists a cocompact manifold model for $\uE H$.
 \item $\ugd G < \infty$.
 \item There exists no manifold model for $\uE G$.
\end{enumerate} 
\end{Theorem}

\begin{proof}[Proof of Theorem \ref{theorem:BPD groups with no manifold model}]
Let $G$ be one of the groups constructed by Block and Weinberger in the theorem above.  Since $H$ has a cocompact model for $\uE H$ it has finitely many conjugacy classes of finite subgroups hence $G$ has finitely many conjugacy classes of finite subgroups, since all torsion in $G$ is contained in $H$.  Let $K$ be a finite subgroup of $G$, so $K$ is necessarily a subgroup of $H$ and the normaliser $N_HK$ is finite index in $N_GK$.  Since there is a cocompact model for $\uE H$, the normaliser $N_HK$ is $\FP_\infty$ and finitely presented \cite[Theorem 0.1]{LuckMeintrup-UniversalSpaceGrpActionsCompactIsotropy} hence $N_GK$ and $W_GK$ are $\FP_\infty$ and finitely presented too \cite[VIII.5.1]{Brown}\cite[2.2.5]{Robinson}.  Using Lemma \ref{lemma:OFFPn characterisation}, $G$ is of type $\uFP$.  

Finally, using \cite[III.(6.5)]{Brown}, there is a chain of isomorphisms for all natural numbers $i$,
\begin{align*}
H^i(W_GK, R[W_GK]) &\cong H^i(N_GK, R[N_GK]) \\
&\cong H^i(N_HK, R[N_HK]) \\
&\cong H^i(W_HK, R[W_HK)
\end{align*}
proving that the Weyl groups of finite subgroups have the correct cohomology.
\end{proof}

\begin{Remark}
 Although it doesn't appear in the statements of \cite[Theorems 1.5, 1.8]{BlockWeinberger-GeneralizedNielsenRealizationProblem}, Block and Weinberger do prove that there is a cocompact manifold model for $\uE G$, in their notation this is the space $\widetilde{X}$.
\end{Remark}

\subsection{Actions on \texorpdfstring{$R$}{R}-homology manifolds}\label{subsection:actions on homology manifolds}

Following \cite{DicksLeary-SubgroupsOfCoxeterGroups} we define an \emph{$R$-homology $n$-manifold} to be a locally finite simplicial complex $M$ such that the link $\sigma$ of every $i$-simplex of $M$ satisfies
\[ H_j(\sigma, R) = \left\{ \begin{array}{l l} R & \text{ if } j = n - i - 1 \text{ or } j = 0 \\ 0 & \text{ else.} \end{array} \right. \]
for all $i$ such that $n - i - 1 \ge 0$ and the link is empty if $n - i - 1 < 0$.  In particular $M$ is an $n$-dimensional simplicial complex.  $M$ is called \emph{orientable} if we can choose an orientation for each $n$-simplex which is consistent along the $(n-1)$-simplices and we say that $M$ is \emph{$R$-orientable} if either $M$ is orientable or if $R$ has characteristic $2$.

A topological space $X$ is called $R$-acyclic if the reduced homology $\tilde{H}_*(X, R)$ is trivial.

As for smooth manifolds, we have the following theorem:

\begin{Theorem}\label{theorem:proper actions on Racyclic Rhomology then BPD}
 If $G$ is a group acting properly and cocompactly on an $R$-acyclic $R$-orientable $R$-homology $n$-manifold $M$ then
\[ H^i(G, RG) \cong \left\{ \begin{array}{l l} R & \text{ if } i = n \\ 0 & \text{ else.}\end{array} \right. \]
\end{Theorem}
\begin{proof}
 By \cite[Lemma F.2.2]{Davis} $ H^i(G, RG) \cong H^i_c(M, R) $, where $H^i_c$ denotes cohomology with compact supports.  By Poincar\'e duality for $R$-orientable $R$-homology manifolds (see for example \cite[Theorem 5]{DicksLeary-SubgroupsOfCoxeterGroups}), there is a duality isomorphism $H^i_c(M, R) \cong H_{n-i}(M, R) $.  Finally, since $M$ is assumed acyclic,
\[ H_{n-i}(M, R) \cong \left\{ \begin{array}{l l} R & \text{ if } i = n \\ 0 & \text{ else.}\end{array} \right.\]
\end{proof}

\begin{Example}\label{example:DicksLeary example PD over R not Z}
In \cite[Example 3]{DicksLeary-SubgroupsOfCoxeterGroups}, Dicks and Leary construct a group which is Poincar\'e duality over $R$, arising from an action on an $R$-orientable $R$-acyclic $R$-homology manifold, but which is not Poincar\'e duality over $\ZZ$.  Here $R$ may be any ring for which a fixed prime $q$ is invertible, for example $R = \FF_p$ for $p \neq q$ or $R = \QQ$.
\end{Example}

\begin{Cor}\label{cor:actions on Rorient Rhom manifolds gives BPD}
 If $G$ is a group which admits a cocompact model $X$ for $\uE G$ such that for every finite subgroup $H$ of $G$, $X^H$ is an $R$-orientable $R$-homology manifold.  Then $G$ is Bredon--Poincar\'e duality over $R$.
\end{Cor}

\begin{Remark}
 In the case $R = \ZZ$ we can drop the condition that $M$ be orientable since this is implied by being acyclic.  This is because if $M$ is acyclic then $\pi_1(M)$ is perfect, thus $\pi_1(M)$ has no normal subgroups of prime index, in particular $M$ has no index $2$ subgroups.  But if $M$ were non-orientable then the existence of an orientable double cover (see for example \cite[p.234]{Hatcher}) would imply that $\pi_1(M)$ has a subgroup of index $2$.
\end{Remark}

Let $p$ be a prime and $\FF_p$ the field of $p$ elements.  A consequence of Smith theory \cite[III]{Bredon-IntroductionToCompactTransformationGroups} is the following theorem.

\begin{Theorem}\label{thm:smith}
 If $G$ is a finite $p$-group acting properly on an $\FF_p$-homology manifold $M$ then the fixed point set $M^G$ is also an $\FF_p$-homology manifold.  If $p \neq 2$ then $M^G$ has even codimension in $M$.
\end{Theorem}

\begin{Cor}[Actions on homology manifolds]\label{cor:actions on hom manifolds}~
\begin{enumerate}
 \item Let $G$ have an $n$-dimensional $\FF_p$ homology manifold model $M$ for $\uE G$.  If $H$ is a finite $p$-subgroup of $G$ then $M^H$ is an $\FF_p$-homology manifold.  In particular if all finite subgroups of $G$ are $p$-groups then $G$ is Bredon--Poincar\'e duality.  If $ p \neq 2$ and $H$ is a finite $p$-subgroup of $G$ then $n - n_H$ is even.
 \item Let $G$ have an $n$-dimensional $\ZZ$-homology manifold model $M$ for $\uE G$.  If $p \neq 2$ is a prime and $H$ is a finite $p$-subgroup of $G$ such that $M^H$ is a $\ZZ$-homology manifold then $n - n_H$ is even.
\end{enumerate}
\end{Cor}

\begin{Remark}
 Given a group $G$ with subgroup $H$ which is not of prime power order, looking at the Sylow $p$-subgroups can give further restrictions.  For example if $P_i$ for $i \in I$ is a set of Sylow $p$-subgroups of $H$, one for each prime $P$, then by \cite[Ex. 4.10]{Rotman-Groups} $G$ is generated by the $P_i$.  Thus if $G$ acts on an $R$-homology manifold then the fixed points of $H$ are exactly the intersection of the fixed points of the $P_i$.  
\end{Remark}

\subsection{One Relator Groups}

The following lemma is adapted from \cite[5.2]{BieriEckmann-HomologicalDualityGeneralizingPoincareDuality}.
\begin{Lemma}\label{lemma:tech lemma for one-relator}
 If $G$ is $\FP_2$ with $\cd G = 2$ and $H^1(G, \ZZ G) = 0$ then $G$ is duality.
\end{Lemma}
\begin{proof}
We must show that $H^2(G, \ZZ G)$ is a flat $\ZZ$-module.  Consider the short exact sequence of $\ZZ G$ modules
\[0 \longrightarrow \ZZ G \stackrel{\times p}{\longrightarrow} \ZZ G \longrightarrow \FF_pG \longrightarrow 0\]
This yields a long exact sequence
\[\cdots \longrightarrow H^1(G, \FF_p G) \longrightarrow H^2(G, \ZZ G) \stackrel{\times p}{\longrightarrow} H^2(G, \ZZ G) \longrightarrow \cdots  \]
By \cite[Corollary 3.6]{Bieri-HomDimOfDiscreteGroups}, $H^1(G, \FF_p G) \cong  H^1(G, \ZZ G) \otimes_\ZZ \FF_p = 0$.  Hence the map $H^2(G, \ZZ G) \stackrel{\times p}{\longrightarrow} H^2(G, \ZZ G)$ must have zero kernel for all $p$, in other words $H^2(G, \ZZ G)$ is torsion-free.  But the torsion-free $\ZZ$-modules are exactly the flat $\ZZ$-modules.
\end{proof}

Let $G$ be a one-relator group (see \cite[\S 5]{LyndonSchupp} for background on these groups), then:
\begin{enumerate}
 \item $G$ is $\uFP$ and $\ucd_\ZZ G = 2$ \cite[4.12]{Luck-SurveyOnClassifyingSpaces}.
 \item $G$ contains a torsion-free subgroup $Q$ of finite index \cite{FischerKarrassSolitar-OneRelatorGroupsHavingElementsOfFiniteOrder}.
\end{enumerate}
If $\cd_{\ZZ} Q \le 1$ then $Q$ is either finite or a finitely generated free group so $G$ is either finite or virtually finitely generated-free.  Thus $G$ is Bredon duality over $\ZZ$ by \ref{lemma:duality D0}, \ref{remark:duality uD1}, and \ref{prop:duality uD1 Rtf}.  Assume therefore that $\cd_{\ZZ} Q = 2$.  Being finite index in $G$, $Q$ is also $\FP_2$ and $H^1(Q, \ZZ Q) = H^1(G, \ZZ G) = 0$ \cite[III.(6.5)]{Brown}, thus by Lemma \ref{lemma:tech lemma for one-relator} $Q$ is a duality group and $G$ is virtual duality.  

Every finite subgroup of $G$ is subconjugated to a finite cyclic self-normalising subgroup $C$ of $G$ \cite[5.17,5.19]{LyndonSchupp}, and furthermore the normaliser of any finite subgroup is subconjugate to $C$---if $K$ is a non-trivial subgroup of $C$ and $n \in N_G K$ then $n^{-1}C n \cap C \neq 1$ and \cite[5.19]{LyndonSchupp} implies that $n \in C$.  For an arbitrary non-trivial finite subgroup $K^\prime$, since $K^\prime$ is conjugate to some $K \le C$, the normaliser $N_GK^\prime$ is conjugate to $N_GK \le C$.

Since the normaliser of any non-trivial finite subgroup $F$ is finite, 
\[H^i(N_GF, \ZZ [N_G F]) = \left\{\begin{array}{l l}0 & \text{ if $i > 0$,} \\ \ZZ & \text{ if $i = 0$.} \end{array} \right.\]
Hence $G$ is Bredon duality of dimension $2$.  In summary:
 
\begin{Prop}
 If $G$ is a one relator group with $H^1(G, \ZZ G) = 0$ then $G$ is Bredon duality over any ring $R$.
\end{Prop}

\begin{Remark}
If $G$ is a one relator group with $H^1(G, \ZZ G) \neq 0$ then, since $G$ is $\uFP_0$, $G$ has bounded orders of finite subgroups by Lemma \ref{lemma:OFFPn characterisation}.  By a result of Linnell, $G$ admits a decomposition as the fundamental group of a finite graph of groups with finite edge groups and vertex groups $G_v$ satisfying $H^1(G, \ZZ G) = 0$ \cite{Linnell-OnAccessibilityOfGroups}.  These vertex groups are subgroups of virtually torsion-free groups so in particular virtually torsion-free with $\ucd_\ZZ G \le 2$.  Lemma \ref{lemma:duality FP2 split is FP2} below gives that the vertex groups are $\FP_2$ and Lemma \ref{lemma:tech lemma for one-relator} shows that the edge groups are virtually duality.  
\end{Remark}

\begin{Lemma}\label{lemma:duality FP2 split is FP2}
 Let $G$ be a group which splits as a finite graph of groups with finite edge groups $G_e$, indexed by $E$, and vertex groups $G_v$, indexed by $V$.  Then if $G$ is $\FP_2$, so are the vertex groups $G_v$.
\end{Lemma}
\begin{proof}
 Fix a vertex group $G_v$.  Let $M_\lambda$, for $\lambda \in \Lambda$, be a directed system of $\ZZ G_v$ modules with $\varinjlim M_\lambda = 0$.  To use the Bieri-Eckmann criterion \cite[Theorem 1.3]{Bieri-HomDimOfDiscreteGroups}, we must show that $\varinjlim H^i(G_v, M_\lambda) = 0$ for $i = 1,2$.
 
 The Mayer-Vietoris sequence associated to the graph of groups is 
 \[ \cdots \longrightarrow H^i(G, -) \longrightarrow \bigoplus_{v \in V} H^i(G_v, -) \longrightarrow \bigoplus_{e \in E} H^i(G_e, -) \longrightarrow \cdots \]
 Now $\varinjlim M_\lambda = 0$, so $\varinjlim \Ind_{\ZZ G_v}^{\ZZ G} M_\lambda = 0$ as well.  Evaluating the Mayer-Vietoris sequence at $\Ind_{\ZZ G_v}^{\ZZ G} M_\lambda$, taking the limit, and using the Bieri-Eckmann criterion, implies 
 \[\varinjlim_{\Lambda} \bigoplus_{v \in V} H^i(G_v, \Ind_{\ZZ G_v}^{\ZZ G} M_\lambda ) = 0\]
In particular $\varinjlim H^i(G_v, \Ind_{\ZZ G_v}^{\ZZ G} M_\lambda ) = 0$ and because, as $\ZZ G_v$-module, $M_\lambda$ is a direct summand of $\Ind_{\ZZ G_v}^{\ZZ G} M_\lambda$ \cite[VII.5.1]{Brown}, this implies $\varinjlim H^i(G_v, M_\lambda ) = 0$.  
\end{proof}

\subsection{Discrete Subgroups of Lie Groups}
If $L$ is a Lie group with finitely many path components, $K$ a maximal compact subgroup and $G$ a discrete subgroup then $L/K$ is a model for $\uE G$.  The space $L/K$ is a manifold and the action of $G$ on $L/K$ is smooth so the fixed point subsets of finite groups are submanifolds of $L/K$, using Lemma \ref{Lemma:proper and locally linear => fixed points are submanifolds}.  If we assume that the action is cocompact then $G$ is seen to be of type $\uFP$, $\ucd G = \dim L/K$ and $G$ is a Bredon--Poincar\'e duality group.  See \cite[Theorem 5.24]{Luck-SurveyOnClassifyingSpaces} for a statement of these results.

\begin{Example}\label{example:disc sub of lie not vtf}
In \cite{Raghunathan-TorsionInCoCompactLatticesInSpin}\cite{Raghunathan-CorrigendumTorsionInCoCompactLatticesInSpin}, examples of cocompact lattices in finite covers of the Lie group $\text{Spin}(2,n)$ are given which are not virtually torsion-free.  
\end{Example}

\subsection{Virtually Soluble Groups}
The observation below appears in \cite[Example 5.6]{MartinezPerez-EulerClassesAndBredonForRestrictedFamilies}.

If $G$ is virtually soluble duality group $G$ then $G$ is $\uFP$ and $\ucd G = hG$, where $hG$ denotes the Hirsch length of $G$ \cite{MartinezPerezNucinkis-VirtuallySolubleGroupsOfTypeFPinfty}.  We claim $G$ is also Bredon duality, so we must check the cohomology condition on the Weyl groups.  Since $G$ is $\uFP$, the normalisers $N_G F$ of any finite subgroup $F$ of $G$ are $\FP_\infty$ (Lemma \ref{lemma:OFFPn characterisation}). Subgroups of virtually-soluble groups are virtually-soluble \cite[5.1.1]{Robinson}, so the normalisers $N_GF$ are virtually-soluble $\FP_\infty$ and hence virtually duality \cite{Kropholler-CohomologicalDimensionOfSolubleGroups}, and so the Weyl groups satisfy the required condition on cohomology.

Additionally, if $G$ is a virtually soluble Poincar\'e duality group then we claim $G$ is Bredon--Poincar\'e duality.  By \cite[Theorem 9.23]{Bieri-HomDimOfDiscreteGroups}, $G$ is virtually-polycyclic .  Subgroups of virtually-polycyclic groups are virtually-polycyclic \cite[p.52]{Robinson}, so $N_GF$ is polycyclic $\FP_\infty$ for all finite subgroups $F$ and, since polycylic groups are Poincar\'e duality,
\[H^{n_F}(N_G F, \ZZ [N_G F]) = \ZZ\]

\begin{Prop}\label{prop:bredon equivalent to soluble virtual duality}The following conditions on a virtually-soluble group $G$ are equivalent:
\begin{enumerate}
 \item $G$ is $\FP_\infty$.
 \item $G$ is virtually duality.
 \item $G$ is virtually torsion-free and $\vcd G = hG < \infty$.
 \item $G$ is Bredon duality.
\end{enumerate}
Additionally, if $G$ is Bredon duality then $G$ is virtually Poincar\'e duality if and only if $G$ is virtually-polycyclic if and only if $G$ is Bredon--Poincar\'e duality.
\end{Prop}
\begin{proof}
 The equivalence of the first three is \cite{Kropholler-CohomologicalDimensionOfSolubleGroups} and \cite{Kropholler-OnGroupsOfTypeFP_infty}.  That $(4) \Rightarrow (1)$ is obvious and $(1) \Rightarrow (4)$ is the discussion above.
\end{proof}

\subsection{Elementary Amenable Groups}

If $G$ is an elementary amenable group of type $\FP_\infty$ then $G$ is virtually soluble \cite[p.4]{KMN-CohomologicalFinitenessConditionsForElementaryAmenable}, in particular Bredon duality over $\ZZ$ of dimension $hG$.  The converse, that every elementary amenable Bredon duality group is $\FP_\infty$, is obvious.

If $G$ is elementary amenable $\FP_\infty$ then the condition $H^n(G, \ZZ G) \cong \ZZ$ implies that $G$ is Bredon--Poincar\'e duality, so for all finite subgroups $H^{n_F}(N_GF, \ZZ N_GF) \cong \ZZ$.  A natural question is whether 
\[ H^{n_F}(N_G F, \ZZ [N_G F]) = \ZZ \]
can ever occur for an elementary amenable, or indeed a soluble Bredon-duality, but not Bredon--Poincar\'e duality group.  An example of this behaviour is given below.

\begin{Example}
We construct a finite index extension of the Baumslag-Solitar group $BS(1,p)$, for $p$ a prime.
\[BS(1,p) = \langle x,y  :  y^{-1}xy = x^p\rangle \]
This has a normal series (for an explanation see \cite[p.60]{LennoxRobinson}):
\[1 \unlhd \langle x \rangle \unlhd \langle \langle x \rangle \rangle \unlhd  BS(1,p)\]
Whose quotients are $\langle x \rangle / 1 \cong \ZZ$, $\langle \langle x \rangle \rangle / \langle x \rangle \cong C_{p^\infty}$ and $BS(1,p)/\langle \langle x \rangle \rangle \cong \ZZ$, where $C_{p^\infty}$ denotes the Pr\"ufer group (see \cite[p.94]{Robinson} for a definition).  Clearly $BS(1,p)$ is finitely generated torsion-free soluble with $hBS(1,p) = 2$, but not polycyclic, since $C_{p^\infty}$ does not have the maximal condition on subgroups \cite[5.4.12]{Robinson}, thus $BS(1,p)$ is not Poincar\'e duality.  Also since $BS(1,p)$ is an HNN extension of $\langle x \rangle \cong \ZZ$ it has cohomological dimension $2$ \cite[Proposition 6.12]{Bieri-HomDimOfDiscreteGroups} and thus $\cd BS(1,p) = hBS(1,p)$.  By Proposition \ref{prop:bredon equivalent to soluble virtual duality}, $BS(1,p)$ is a Bredon duality group.

Recall that elements of $BS(1,p)$ can be put in a normal form: $y^i x^k y^{-j} $ where $i,j\ge 0$ and if $i,j > 0$ then $n \nmid k$.  Consider the automorphism $\varphi$ of $BS(1,p)$, sending $x \mapsto x^{-1}$ and $y \mapsto y$, an automorphism since it is its own inverse and because the relation $ y^{-1}xy = x^p $ in $BS(1,p)$ implies the relation $y^{-1}x^{-1}y = x^{-p} $.  Let  $y^i x^k, y^{-j} $ be an element in normal form.
\[\varphi :y^i x^k y^{-j} \longmapsto y^i x^{-k} y^{-j}  \]
So the only fixed points of $\varphi$ are in the subgroup $\langle y \rangle \cong \ZZ$.  Form the extension 
\[1 \longrightarrow BS(1,p) \longrightarrow G \longrightarrow C_2 \longrightarrow 1 \]
Where $C_2$ acts by the automorphism $\varphi$.  The property of being soluble is extension closed \cite[5.1.1]{Robinson}, so $G$ is soluble virtual duality and Bredon duality by Proposition \ref{prop:bredon equivalent to soluble virtual duality}.  The normaliser 
\[ N_G C_2 = C_G C_2 = \{ g \in G  :  gz = zg \text{ for the generator }z \in C_2 \}\]
is the points in $G$ fixed by $\varphi$, so $C_GC_2 \cong \ZZ$.  Thus $N_GC_2$ is virtually-$\ZZ$ and $H^1(N_GC_2, \ZZ[N_GC_2]) \cong \ZZ$ \cite[13.5.5]{Geoghegan}.  Since $BS(1,p)$ is finite index in $G$, by \cite[III.(6.5)]{Brown}
\[H^2(G, \ZZ G) \cong H^2(BS(1,p), \ZZ [BS(1,p)]) \]
However since $BS(1,p)$ is not Poincar\'e duality, $ H^n(BS(1,p), \ZZ [BS(1,p)])$ is $\ZZ$-flat but not isomorphic to $\ZZ$.
\end{Example}

\begin{Remark}
Baues \cite{Baues-Infrasolvmanifolds} and Dekimpe \cite{Dekimpe-PolycyclicGroupAdmitsNILAffine} proved independently that any virtually polycyclic group $G$ can be realised as a NIL affine crystallographic group---$G$ acts properly, cocompactly, and by diffeomorphisms on a simply connected nilpotent Lie group of dimension $hG$.  Any connected, simply connected nilpotent Lie group is diffeomorphic to some Euclidean space \cite[I\S 16]{Knapp} and hence contractible, so any elementary amenable Bredon--Poincar\'e duality group has a manifold model for $\uE G$.
\end{Remark}

\section{The Reflection Group Trick}\label{section:reflection groups}
Recall Corollary \ref{cor:actions on hom manifolds}, that if $G$ has a cocompact $n$-dimensional $\ZZ$-homology manifold model $M$ for $\uE G$ such that all fixed point sets $M^H$ are $\ZZ$-homology manifolds, and $H$ is a finite $p$-subgroup of $G$ with $p \neq 2$ then $n - n_H$ is even.  In this section we construct Bredon--Poincar\'e duality groups $G$ over $\ZZ$ of arbitrary dimension such that, for any fixed prime $p \neq 2$:
\begin{enumerate}
 \item All of the finite subgroups of $G$ are $p$-groups.
 \item $\mathcal{V}(G)$ is any vector with $n-n_H$ even for all finite subgroups $H$.
\end{enumerate}

The method of constructing these examples was recommended to us by Ian Leary and uses the equivariant reflection group trick of Davis and Leary which we review below.
We only need a specific case of the trick, for a full explanation see \cite[\S 2]{DavisLeary-DiscreteGroupActionsOnAsphericalManifolds} or \cite[\S 11]{Davis}.  Note that we write ``$\Gamma$'' instead of ``$W$'', as is used in \cite{DavisLeary-DiscreteGroupActionsOnAsphericalManifolds}, to denote a Coxeter group so the notation can't be confused with our use of $W_GH$ for the Weyl group.

Let $M$ be an compact contractible $n$-manifold with boundary $\partial M$, such that $\partial M$ is triangulated as a flag complex.  Let $G$ be a group acting on $M$ such that the induced action on the boundary is by simplicial automorphisms.  Let $\Gamma$ be the right angled Coxeter group corresponding to the flag complex $\partial M$, the group $G$ acts by automorphisms on $\Gamma$ and we can form the semi-direct product $\Gamma \rtimes G$.  Moreover there is a space $\mathcal{U} = \mathcal{U}(M, \partial M, G)$ such that 
\begin{enumerate}
 \item $\mathcal{U}$ is an contractible $n$-manifold without boundary.
 \item $\Gamma \rtimes G$ acts properly and cocompactly on $\mathcal{U}$.
 \item For any finite subgroup $H$ of $G$, we have $\mathcal{U}^H = \mathcal{U}(M^H, (\partial M)^H, W_GH)$, in particular $\dim \mathcal{U}^H = \dim M^H$.
 \item $W_{\Gamma \rtimes G}H = \Gamma^H \rtimes W_GH$, where $\Gamma^H$ is the right-angled Coxeter group associated to the flag complex $(\partial M)^H$.
 \item If $M$ is the cone on a finite complex then $\mathcal{U}$ has a CAT(0) cubical structure such that the action of $\Gamma \rtimes G$ is by isometries.  
\end{enumerate}

Every Coxeter group contains a finite-index torsion-free subgroup \cite[Corollary D.1.4]{Davis}, let $\Gamma^\prime$ denote such a subgroup of $\Gamma$ and assume that $\Gamma^\prime$ is normal.  Then $\Gamma^\prime \rtimes G$ is finite index in $\Gamma \rtimes G$ and so acts properly and cocompactly on $\mathcal{U}$ also.

A CAT(0) cubical complex has a CAT(0) metric \cite[Remark 2.1]{Wise-FromRichesToRaags} and any contractible CAT(0) space is a model for $\uE G$ \cite[Corollary II.2.8]{BridsonHaefliger} (see also \cite[Theorem 4.6]{Luck-SurveyOnClassifyingSpaces}).  This implies the following:

\begin{Lemma}\label{lemma:arbVG construction is uEG}
 If $M$ is the cone on a finite complex then $\mathcal{U}$ is a cocompact model for $\uE (\Gamma \rtimes G)$ and for $\uE (\Gamma^\prime \rtimes G)$.  In particular, $\Gamma^\prime \rtimes G$ is of type $\uFP$.
\end{Lemma}

\begin{Lemma}\label{lemma:arbVG construction, conj of finite subgroups}
 Let $M$ be the cone on a finite complex and assume that $G$ acts on $M$ with a fixed point.  If $K$ is a finite subgroup of $\Gamma \rtimes G$ then $K$ is subconjugate to $G$.
\end{Lemma}
\begin{proof}
Since $\mathcal{U}$ is a model for $\uE (\Gamma \rtimes G)$, the finite subgroup $K$ necessarily fixes a vertex of $\mathcal{U}$ and hence is subconjugate to the point stabiliser.  Examining the construction of $\mathcal{U}$ shows that the stabilisers of the vertices are all subconjugate to $G$ and the result follows.
\end{proof}

\begin{Theorem}\label{theorem:arbVG general setup}
Let $G$ be a finite group with real representation $\rho : G \longhookrightarrow \text{GL}_n \RR$ and, for all subgroups $H$ of $G$, let $n_H$ denote the dimension of the subspace of $\RR^n$ fixed by $H$.  Then there exists a Bredon--Poincar\'e duality group $\Gamma^\prime \rtimes G$ of dimension $n$ such that:
\[\mathcal{V} (\Gamma^\prime \rtimes G) = \{ n_H  :  H \le G\}\]
\end{Theorem}
\begin{proof}
 Restrict $\rho$ to an action on $(\DD^n, S^{n-1})$ and choose a flag triangulation of $S^{n-1}$ respecting the action of $G$ (use for example \cite{Illman-SmoothEquiTriangulationsForGFinite}).  Now use the reflection group trick to obtain a Coxeter group $\Gamma$, normal finite-index torsion-free subgroup $\Gamma^\prime$ and space $\mathcal{U}$.  Lemma \ref{lemma:arbVG construction is uEG} gives that $\Gamma^\prime \rtimes G$ is of type $\uFP$.  

 Since $G$ has an $n$-dimensional model for $\uE G$ we have that $\ugd G \le n$ and by Lemma \ref{Lemma:observations}(1) $\cd_\QQ G = n_{1} = n$.  Using the chain of inequalities $n = \cd_\QQ G \le \ucd G \le \ugd G \le n$ shows that $\ucd G = n$.  It remains only to check the condition on the cohomology of the Weyl groups of the finite subgroups.
 
 For any finite subgroup $H$ of $G$, the Weyl group $W_{\Gamma^\prime \rtimes G} H$ acts properly and cocompactly on $\mathcal{U}(M^H, (\partial M)^H, W_G H)$ which is a contractible $n_H$-manifold without boundary.  By Corollary \ref{cor:actions on Rorient Rhom manifolds gives BPD} we have that 
 \[ H^n(W_{\Gamma^\prime \rtimes G} H, \ZZ[W_{\Gamma^\prime \rtimes G} H]) = \left\{ \begin{array}{l l}\ZZ & \text{ if } i = n_H \\ 0 & \text{else} \end{array} \right.  \]
 
 If $K$ is any finite subgroup of $\Gamma^\prime \rtimes G$ then, by Lemma \ref{lemma:arbVG construction, conj of finite subgroups}, $K$ is conjugate in $\Gamma \rtimes G$ to some $H \le G$.  In particular the normalisers of $H$ and $K$ in $\Gamma \rtimes G$ are isomorphic.  Also, since $\Gamma^\prime \rtimes G$ is finite index in $\Gamma \rtimes G$, we have that $N_{\Gamma^\prime \rtimes G}K$ is finite index in $N_{\Gamma \rtimes G}K$.  Calculating the cohomology:
\begin{align*} 
H^n(N_{\Gamma^\prime \rtimes G}K, \ZZ[N_{\Gamma^\prime \rtimes G}K]) 
&\cong H^n(N_{\Gamma \rtimes G}K, \ZZ[N_{\Gamma \rtimes G}K]) \\
&\cong H^n(N_{\Gamma \rtimes G}H, \ZZ[N_{\Gamma \rtimes G}H]) \\
&\cong H^n(N_{\Gamma^\prime \rtimes G}H, \ZZ[N_{\Gamma^\prime \rtimes G}H])
\end{align*}
From the short exact sequence
\[ 0 \longrightarrow K \longrightarrow N_{\Gamma^\prime \rtimes G} K \longrightarrow W_{\Gamma^\prime \rtimes G} K \longrightarrow 0\]
and \cite[Proposition 2.7]{Bieri-HomDimOfDiscreteGroups} we see that $H^n(N_{\Gamma^\prime \rtimes G}H, \ZZ[N_{\Gamma^\prime \rtimes G}H])$ is isomorphic to $H^n(W_{\Gamma^\prime \rtimes G}K, \ZZ[W_{\Gamma^\prime \rtimes G}K])$, thus 

\[ H^n(W_{\Gamma^\prime \rtimes G} K, \ZZ[W_{\Gamma^\prime \rtimes G} K]) = \left\{ \begin{array}{l l}\ZZ & \text{ if } i = n_H \\ 0 & \text{else} \end{array} \right.  \]
\end{proof}

\begin{Example}\label{example:arb VG for pgroups}
We construct a group with the properties mentioned at the beginning of this section.  It will be of the form $G = \Gamma^\prime \rtimes C_{p^m}$, where $C_{p^m}$ is the cyclic group of order $p^m$.

For $i$ between $1$ and $m$ let $w_i$ be any collection of positive integers and let $n = \sum_i 2w_i$. 
If $c$ is a generator of the cyclic group $C_{p^m}$, then $C_{p^m}$ embeds into the orthogonal group $O(n)$ via the real representation
 \begin{align*}
 \rho: C_{p^m} &\longhookrightarrow O(n) \\
 c &\longmapsto (R_{2 \pi/p})^{\oplus w_1} \oplus (R_{2 \pi/p^2})^{\oplus  w_2} \oplus \cdots \oplus (R_{2 \pi/p^m})^{\oplus w_m} 
 \end{align*}
Where $R_\theta$ is the the $2$-dimensional rotation matrix of angle $\theta$.  The image is in $O(n)$ since we chose $n$ such that $2w_1 + \cdots 2w_n = n$.
 
If $i$ is some integer between $1$ and $m$ then there is a unique subgroup $C_{p^{m-i+1}}$ of $C_{p^m}$ with generator $c^{p^{i}}$, in fact this enumerates all subgroups of $C_{p^m}$ except the trivial subgroup.  Under $\rho$, this generator maps to
\[ \rho : c^{p^{i}} \longmapsto 0 \oplus \cdots \oplus 0 \oplus \left( R_{p^{i} 2 \pi / p^{i+1} } \right)^{\oplus w_{i+1}} \oplus \cdots \oplus \left( R_{p^{i} 2 \pi / p^{m} } \right)^{\oplus w_n}\]
In other words, the fixed point set corresponding to $C_{p^{m-i+1}}$ is $\RR^{2w_1 + \cdots + 2w_i}$.  Thus the set of dimensions of the fixed point subspaces of non-trivial finite subgroups of $C_{p^m}$ are
 \[\{2w_1, 2(w_1 + w_2), \ldots, 2(w_1 + w_2 + \ldots + w_{m-1}) \} \]
Applying Theorem \ref{theorem:arbVG general setup} gives a group $\Gamma^\prime \rtimes C_{p^m}$ of type $\uFP$ with 
\[\ucd G = n = \sum^m_{i=1} 2w_i\]
and such that 
\[\mathcal{V}( \Gamma^\prime \rtimes C_{p^m} ) = \{2w_1, 2(w_1 + w_2), \ldots, 2(w_1 + w_2 + \ldots + w_{m-1}) \} \]
Since there were no restrictions on the integers $w_i$, using this technique we can build an even dimensional Bredon--Poincar\'e duality group with any $\mathcal{V}(G)$, as long as all $n_H$ are even.

The case $n$ is odd reduces to the case $n$ is even.  Proposition \ref{prop:duality direct products of BPD groups} shows that if a group $G$ is Bredon--Poincar\'e duality then taking the direct product with $\ZZ$ gives a Bredon--Poincar\'e duality group $G \times \ZZ$ where
\[\mathcal{V}(G \times \ZZ) \cong \{ v + 1  :  v \in \mathcal{V}(G)\} \]
Thus we can build a group with odd $n$ and $\mathcal{V}$ containing only odd elements by building a group with even $n$ and then taking a direct product with $\ZZ$.
\end{Example}

\section{Low Dimensions}\label{section:low dimensions}

This section is devoted to the study of Bredon duality groups and Bredon--Poincar\'e duality groups of low dimension.  We completely classify those of dimension $0$ in Lemma \ref{lemma:duality D0}.  We partially classify those of dimension 1---see Propositions \ref{prop:duality uD1 Rtf} and \ref{prop:duality uPD1 Rtf}, and Question \ref{question:duality D1 not Rtf}---and there is a discussion of the dimension $2$ case.

Recall that a group $G$ is duality of dimension $0$ over $R$ if and only if $\lvert G \rvert$ is finite and invertible in $R$ \cite[Proposition 9.17(a)]{Bieri-HomDimOfDiscreteGroups}.

\begin{Lemma}\label{lemma:duality D0}
 $G$ is Bredon duality of dimension $0$ over $R$ if and only if $\vert G \vert$ is finite.  Any such group is necessarily Bredon--Poincar\'e duality.
\end{Lemma}
\begin{proof}
 By \cite[13.2.11]{Geoghegan}
\[H^0(G, R G) = \left\{ \begin{array}{l l} R & \text{if $\vert G \vert$ is finite} \\ 0 & \text{else.} \end{array}\right.\]
Hence if $G$ is Bredon duality of dimension $0$ then $G$ is finite and moreover $G$ is Bredon--Poincar\'e duality.
 
Conversely, if $G$ is finite then $\ucd_R G = 0$ and $G$ is $\uFP_\infty$ over $R$.  Finally the Weyl groups of any finite subgroup will be finite so by  \cite[13.2.11,13.3.1]{Geoghegan}.
\[H^n(WH, R [WH]) = \left\{ \begin{array}{c c} R & \text{ if } n = 0 \\ 0 & \text{ if } n > 0 \end{array}\right.\]
Thus $G$ is Bredon--Poincar\'e duality of dimension $0$.
\end{proof}

The duality groups of dimension $1$ over $R$ are exactly the groups of type $\FP_1$ over $R$ (equivalently finitely generated groups \cite[Proposition 2.1]{Bieri-HomDimOfDiscreteGroups}) with $\cd_R G = 1$ \cite[Proposition 9.17(b)]{Bieri-HomDimOfDiscreteGroups}.

\begin{Prop}\label{prop:duality uD1 Rtf}
If $G$ is infinite $R$-torsion free, then the following are equivalent:
\begin{enumerate}
 \item $G$ is Bredon duality over $R$, of dimension $1$.
 \item $G$ is finitely generated and virtually-free.
 \item $G$ is virtually duality over $R$, of dimension $1$. 
\end{enumerate}
\end{Prop}
\begin{proof}
 That $2 \Rightarrow 3$ is \cite[Proposition 9.17(b)]{Bieri-HomDimOfDiscreteGroups}.  For $3 \Rightarrow 2$, let $G$ be virtually duality over $R$ of dimension $1$, then by \cite{Dunwoody-AccessabilityAndGroupsOfCohomologicalDimensionOne} $G$ acts properly on a tree.  Since $G$ is assumed finitely generated, by \cite[Theorem 3.3]{Antolin-CayleyGraphsOfVirtuallyFreeGroups} $G$ is virtually-free.

 For $1 \Rightarrow 2$, if $G$ is Bredon duality over $R$ of dimension $1$, then $G$ is automatically finitely generated and $\ucd_R G = 1$.  By Lemma \ref{lemma:Rtf then cdRG le ucdRG} $\cd_R G = 1$ so, as above, by \cite{Dunwoody-AccessabilityAndGroupsOfCohomologicalDimensionOne} and \cite[Theorem 3.3]{Antolin-CayleyGraphsOfVirtuallyFreeGroups}, $G$ is virtually-free.

 For $2 \Rightarrow 1$, if $G$ is virtually finitely generated free then $G$ acts properly and cocompactly on a tree, so $G$ is $\uFP$ over $R$ with $\ucd_R G = 1$.  As $G$ is $\uFP$, for any finite subgroup $K$, the normaliser $N_G K$ is finitely generated.  Subgroups of virtually-free groups are virtually-free, so $N_G K$ is virtually finitely generated free, in particular:
\[ H^i(WK, \ZZ[WK]) = H^i(N_GK, \ZZ[N_G K]) = \left\{ \begin{array}{l l} \text{$\ZZ$-flat} & \text{ for }i = n_K \\ 0 & \text{ else.} \end{array} \right. \]
where $n_K = 0$ or $1$.  Thus $G$ is Bredon duality over $\ZZ$ and hence also over $R$.
\end{proof}

\begin{Remark}\label{remark:duality uD1}
 The only place that the condition $G$ be $R$-torsion-free was used was in the implication $1 \Rightarrow 2$, the problem for groups which are not $R$-torsion-free is that the condition $\ucd_R G \le 1$ is not known to imply that $G$ acts properly on a tree.  If we take $R=\ZZ$ then $\ucd_\ZZ G \le 1$ implies $G$ acts properly on a tree by a result of Dunwoody \cite{Dunwoody-AccessabilityAndGroupsOfCohomologicalDimensionOne}.  We conclude that over $\ZZ$, $G$ is Bredon duality of dimension $1$ if and only $G$ is finitely generated virtually free, if and only if $G$ is virtually duality of dimension $1$.
\end{Remark}

\begin{Question}\label{question:duality D1 not Rtf}
 What characterises Bredon-duality groups of dimension $1$ over $R$?
\end{Question}

\begin{Prop}\label{prop:duality uPD1 Rtf}
 If $G$ is infinite then the following are equivalent:
\begin{enumerate}
 \item $G$ is Bredon--Poincar\'e duality over $R$, of dimension $1$.
 \item $G$ is virtually infinite cyclic.
 \item $G$ is virtually Poincar\'e duality over $R$, of dimension $1$. 
\end{enumerate}
\end{Prop}
\begin{proof}
  The equivalence follows from the fact that for $G$ a finitely generated group, $G$ is virtually infinite cyclic if and only if $H^1(G, RG) \cong R$ \cite[13.5.5,13.5.9]{Geoghegan}.
\end{proof}

In dimension $2$ we can only classify Bredon--Poincar\'e duality groups over $\ZZ$.

The following result appears in \cite[Example 5.7]{MartinezPerez-EulerClassesAndBredonForRestrictedFamilies}, but a proof is not given there.

\begin{Lemma}\label{lemma:duality vsurface is uPD2}
 If $G$ is virtually a surface group then $G$ is Bredon--Poincar\'e duality.
\end{Lemma}

\begin{proof}
As $G$ is a virtual surface group, $G$ has finite index subgroup $H$ with $H$ the fundamental group of some closed surface.  Firstly, assume $H = \pi_1(S_g)$ where $S_g$ is the orientable surface of genus $g$.  If $g = 0$ then $S_g$ is the $2$-sphere and $G$ is a finite group, thus $G$ is Bredon--Poincar\'e duality by Lemma \ref{lemma:duality D0}.  We now treat the cases $g = 1$ and $g>1$ separately. 
If $g > 0$ then by \cite[Lemma 4.4(b)]{Mislin-ClassifyingSpacesProperActionsMappingClassGroups} $G$ is $\uFP$ over $\ZZ$ with $\ucd_\ZZ G \le 2$.  If $g > 1$ then, in the same lemma, Mislin shows that the upper half-plane is a model for $\uE G$ with $G$ acting by hyperbolic isometries.  Giving the upper half plane the structure of a Riemannian manifold with the Poincar\'e metric, this action is by isometries and \cite[10.1]{Davis} gives that the fixed point sets are all submanifolds, hence $G$ is Bredon--Poincar\'e duality of dimension $2$.  If $g=1$ then by \cite[Lemma 4.3]{Mislin-ClassifyingSpacesProperActionsMappingClassGroups}, $G$ acts by affine maps on $\RR^2$ so again $\RR^2$ is a model for $\uE G$ whose fixed point sets are submanifolds, and thus $G$ is Bredon--Poincar\'e duality of dimension $2$.  
We conclude that \emph{orientable} virtual Poincar\'e duality groups of dimension $2$ are Bredon--Poincar\'e duality of dimension $2$.

Now we treat the non-orientable case, so $H = \pi_1(T_k)$ where $T_k$ is a closed non-orientable surface of genus $k$.  In particular $T_k$ has Euler characteristic $\chi(T_k) = 2-k$.  $H$ has an index $2$ subgroup $H^\prime$ isomorphic to the fundamental group of the closed orientable surface of Euler characteristic $2\chi(S)$, thus $H^\prime = \pi_1(S_{k-1})$ the closed orientable surface of genus $k-1$.  If $k = 1$ then $H = \ZZ/2$ and $G$ is a finite group, thus Bredon--Poincar\'e duality by Lemma \ref{lemma:duality D0}.  
Assume then that $k > 1$, we are now back in the situation above where $G$ is virtually $S_g $ for $g > 0$ and as such $G$ is Bredon--Poincar\'e duality of dimension $n$, by the previous part of the proof.
\end{proof}

\begin{Prop}\label{lemma:duality vPD2ZZ iff uPD2ZZ}
The following conditions are equivalent:
\begin{enumerate}
 \item $G$ is virtually Poincar\'e duality of dimension $2$ over $\ZZ$.
 \item $G$ is virtually surface.
 \item $G$ is Bredon--Poincar\'e duality of dimension $2$ over $\ZZ$.
\end{enumerate}
\end{Prop}
\begin{proof}
 That $1 \Leftrightarrow 2$ is \cite{Eckmann-PoincareDualityGroupsOfDimensionTwoAreSurfaceGroups} and that  $2 \Rightarrow 3$ is Lemma \ref{lemma:duality vsurface is uPD2}.  The implication $3 \Rightarrow 2$ is provided by \cite[Theorem 0.1]{Bowditch-PlanarGroupsAndTheSeifertConjecture} which states that any $\FP_2$ group with $H^2(G, \QQ G) = \QQ$ is a virtual surface group and hence a virtual Poincar\'e duality group.  If $G$ is Bredon--Poincar\'e duality of dimension $2$ then $H^i(G, \QQ G) = H^i(G, \ZZ G) \otimes \QQ = \QQ$ and $G$ is $\FP_2$ so we may apply the aforementioned theorem.
\end{proof}

The above proposition doesn't extend from Poincar\'e duality to just duality, as demonstrated by \cite[p.163]{Schneebeli-VirtualPropertiesAndGroupExtensions} where an example, based on Higman's group, is given of a Bredon duality group of dimension $2$ over $\ZZ$ which is not virtual duality.  This example is extension of a finite group by a torsion-free duality group of dimension $2$.  In the theorem it is proved that the group is not virtually torsion-free, that it is Bredon duality follows from Proposition \ref{prop:finite-by-uDn is uDn}.

\begin{Question}
  Is there an easy characterisation of Bredon duality, or Bredon--Poincar\'e duality groups, of dimension $2$ over $R$?
\end{Question}

\section{Extensions}\label{subsection:duality extensions}

In the classical case, extensions of duality groups by duality groups are always duality \cite[9.10]{Bieri-HomDimOfDiscreteGroups}.  In the Bredon case the situation is more complex, for example semi-direct products of torsion-free groups by finite groups may not even be $\uFP_0$ \cite{LearyNucinkis-SomeGroupsOfTypeVF}.  Davis and Leary build examples of finite index extensions of Poincar\'e duality groups which are not Bredon duality, although they are $\uFP_\infty$ \cite[Theorem 2]{DavisLeary-DiscreteGroupActionsOnAsphericalManifolds}, and examples of virtual duality groups which are not of type $\uFP_\infty$ \cite[Theorem 1]{DavisLeary-DiscreteGroupActionsOnAsphericalManifolds}.  In \cite{FarrellLafont-FiniteAutomorphismsOfNegCurvedPDGroups}, Farrell and Lafont give examples of prime index extensions of $\delta$-hyperbolic Poincar\'e duality groups which are not Bredon--Poincar\'e duality.  In \cite[\S 5]{MartinezPerez-EulerClassesAndBredonForRestrictedFamilies},
 Martinez-Perez considers $p$-power extensions of duality groups over fields of characteristic $p$, showing that if $Q$ is a $p$-group and $G$ is Poincar\'e duality of dimension $n$ over a field of characteristic $p$ then then $G \rtimes Q$ is Bredon--Poincar\'e duality of dimension $n$.  These results do not extend from Poincar\'e duality groups to duality groups however \cite[\S 6]{MartinezPerez-EulerClassesAndBredonForRestrictedFamilies}.

 The only really tractable cases are direct products of two Bredon duality groups and extensions of the form finite-by-Bredon duality.

\subsection{Direct Products}
\begin{Lemma}\label{lemma:bredon G uFP and H uFP then GxH is uFP}
If $G_1$ and $G_2$ are $\uFP$ over $R$ then $G_1 \times G_2$ is $\uFP$ over $R$ and
\[\ucd_R G_1 \times G_2 \le \ucd_R G_1 + \ucd_R G_2\]
\end{Lemma}
\begin{proof}
That $\ucd_R G_1 \times G_2 \le \ucd_R G_1 + \ucd_R G_2$ is a special case of \cite[3.59]{Fluch}, the proof used involves showing that given projective resolutions $P_*$ of $\uR$ by Bredon modules for $G_1$ and $Q_*$ of $\uR$ by Bredon modules for $G_2$, the total complex of the tensor product double complex is a projective resolution of $\uR$ by projective  Bredon modules for $G_1 \times G_2$ \cite[3.54]{Fluch}.  So to prove that $G_1 \times G_2$ is $\uFP$ it is sufficient to show that if $P_*$ and $Q_*$ are finite type resolutions, then so is the total complex, but this follows from \cite[3.49]{Fluch}.
\end{proof}

\begin{Lemma}
 If $L$ is a finite subgroup of $G_1 \times G_2$ then the normaliser $N_{G_1 \times G_2} L$ is finite index in $N_{G_1} \pi_1L \times N_{G_2} \pi_2 L$, where $\pi_1$ and $\pi_2$ are the projection maps from $G_1 \times G_2$ onto the factors $G_1$ and $G_2$.
\end{Lemma}
\begin{proof}
 It's straightforward to check that
 \[N_{G_1 \times G_2} L \le N_{G_1} \pi_1L \times N_{G_2} \pi_2 L\]
 To see it is finite index, observe that $N_{G_1} \pi_1L \times N_{G_2} \pi_2 L$ acts by conjugation on $(\pi_1 L \times \pi_2 L)/L$, but this set is finite so the stabiliser of $L$, which is exactly $N_{G_1 \times G_2} L $, is finite index in $N_{G_1} \pi_1L \times N_{G_2} \pi_2 L$.
\end{proof}

\begin{Prop}\label{prop:duality direct products of BPD groups}
 If $G_1$ and $G_2$ are Bredon duality (resp. Bredon--Poincar\'e duality), then $G \cong G_1 \times G_2$ is Bredon duality (resp. Bredon--Poincar\'e duality).  Furthermore
\[\mathcal{V}(G_1 \times G_2) = \big\{ v_1 + v_2  :  v_1 \in \mathcal{V}(G_1)\cup \{n_1(G_1) \} \text{ and } v_2 \in \mathcal{V}(G_2)\cup \{n_1(G_2) \} \big\} \]
\end{Prop}
\begin{proof}
 By Lemma \ref{lemma:bredon G uFP and H uFP then GxH is uFP}, $G \times H$ is $\uFP$.  If $L$ is some finite subgroup, the normaliser $N_G L$ is finite index in $N_{G_1} \pi_1L \times N_{G_2} \pi_2 L$ so an application of Shapiro's Lemma \cite[III.(6.5) p.73]{Brown} gives that for all $i$,
\[H^i(N_{G} L, R[N_{G} L] ) \cong H^n( N_{G_1} \pi_1L \times N_{G_2} \pi_2 L, R[N_{G_1} \pi_1L \times N_{G_2} \pi_2 L])\]
Noting the isomorphism of $RG$ modules 
\[ R[N_{G_1} \pi_1L \times N_{G_2} \pi_2 L] \cong R[N_{G_1} \pi_1L] \otimes R[N_{G_2} \pi_2 L] \]
The K\"unneth formula for group cohomology (see \cite[p.109]{Brown}) is
\[\xymatrix@-7pt{
0 \ar[d] \\
 \bigoplus_{i + j = k} \big( H^i(N_{G_1} \pi_1L, R[N_{G_1} \pi_1L]) \otimes H^j(N_{G_1} \pi_1L, R[N_{G_1} \pi_1L]) \big) \ar[d]\\
 H^k(G_1 \times G_2, R[N_{G_1} \pi_1L \times N_{G_2} \pi_2 L]) \ar[d] \\
 \bigoplus_{i + j = k+1} \Tor_1( H^i(G_1,R[N_{G_1} \pi_1L] ), H^j(G_2, R[N_{G_2} \pi_2L]) ) \ar[d] \\
0}\]
Note that here we are using that $R[N_{G_i} \pi_iL]$ is $R$-free.  Since $H^i(G_1,R[N_{G_1} \pi_1L] )$ is assumed $R$-flat the $\Tor_1$ term is zero.  Hence the central term is non-zero only when $i = n_{\pi_1L}$ and $j = n_{\pi_2L}$, in which case it is $R$-flat.  If $G_1$ and $G_2$ are Bredon--Poincar\'e duality then the central term in this case is $R$.
\end{proof}

\subsection{Finite-by-Duality Groups}

Throughout this section, $F$, $G$ and $Q$ will denote groups in a short exact sequence
\[0 \longrightarrow F \longrightarrow G \stackrel{\pi}{\longrightarrow} Q \longrightarrow 0\]
Where $F$ is finite.  This section builds up to the proof of Proposition \ref{prop:finite-by-uDn is uDn} that if $Q$ is Bredon duality of dimension $n$ over $R$, then $G$ is also.

\begin{Lemma}\label{lemma:Finite by Q implies HG is HQ}
$ H^i(G, R G) \cong H^i(Q, R Q ) $ for all $i$ and any ring $R$.
\end{Lemma}
\begin{proof}

 The Lyndon-Hochschild-Serre spectral sequence associated to the extension is:
\[H^p(Q, H^q(F, R G) ) \underset{p}{\Rightarrow} H^{p+q}(G, R G)\]
$R G$ is projective as a $R F$-module so by \cite[Proposition 5.3, Lemma 5.7]{Bieri-HomDimOfDiscreteGroups}, 
\[H^q(F, R G) \cong H^q(F, R F) \otimes_{R F} R G = \left\{ \begin{array}{l l} R \otimes_{R F} R G = R Q & \text{ if } q = 0 \\ 0 & \text{ else.} \end{array} \right\} \]
 The spectral sequence collapses to $H^i(G, R G) \cong H^i(Q, R Q )$.
\end{proof}

\begin{Lemma}\label{lemma:Finite-by-uFP0 is uFP0}
 If $Q$ is $\uFP_0$, then $G$ is $\uFP_0$.
\end{Lemma}
\begin{proof}
Let $B_i$ for $i = 0, \ldots, n$ be a collection of conjugacy class representatives of all finite subgroups in $Q$.  Let $\{B_i^j\}_j$ be a collection of conjugacy class representatives of finite subgroups in $G$ which project onto $B_i$.  Since $F$ is finite $\pi^{-1}(B_i)$ is finite and there are only finitely many $j$ for each $i$, we claim that these $B_i^j$ are conjugacy class representatives for all finite subgroups in $G$.

Let $K$ be some finite subgroup of $G$, we need to check it is conjugate to some $B_i^j$.  $A= \pi(K)$ is conjugate to $B_i$, let $q \in Q$ be such that $q^{-1}Aq = B_i$ and let $g \in G$ be such that $\pi(g) = q$. 

$\pi(g^{-1} K g) = q^{-1} A q = B_i$ so $g^{-1}Kg$ is conjugate to some $B_i^j$ and hence $K$ is conjugate to some $B_i^j$.  Since we have already observed that for each $i = 0,\ldots, n$ the set $\{B_i^j\}_j$ is finite, $G$ has finitely many conjugacy classes of finite subgroups.
\end{proof}

\begin{Lemma}\label{lemma:pi G to Q with finite kernel then NGK is finite index in etc}
If $K$ is a finite subgroup of $G$ then $N_GK$ is finite index in $N_G (\pi^{-1}\circ\pi(K))$.
\end{Lemma}
\begin{proof}
 $N_GK$ is a subgroup of $N_G (\pi^{-1}\circ\pi(K))$ since if $g^{-1}Kg = K$ then 
\[\left(\pi^{-1} \circ \pi(g)\right) \left(\pi^{-1} \circ \pi(K)\right) \left(\pi^{-1} \circ \pi(g)\right)^{-1} = \pi^{-1} \circ \pi(K)\]
But $g \in \pi^{-1} \circ \pi(g)$ so $g\left( \pi^{-1} \circ \pi(K) \right) g^{-1} = \pi^{-1} \circ \pi(K)$.

$N_G K$ is the stabiliser of the conjugation action of $G$ on $G/K$ so by the above can be described as the stabiliser of the action of $N_G \left(\pi^{-1}\circ\pi(K)\right)$ on $ G/K$ by conjugation.  But $N_G \left(\pi^{-1}\circ\pi(K)\right)$ maps $K$ inside $\pi^{-1}\circ\pi(K)$ so $N_G K$ is the stabiliser of $N_G \left(\pi^{-1}\circ\pi(K)\right)$ on  $\pi^{-1}\circ\pi(K)/K$.  

$K$ is finite, so $\pi(K)$ is finite and since the kernel of $\pi$ is finite, $\pi^{-1}\circ\pi(K)$ is finite.  Hence the stabiliser must be a finite index subgroup of $N_G\left(\pi^{-1}\circ\pi(K)\right)$.  
\end{proof}

\begin{Lemma}\label{lemma:pi G to Q and L subgroup of Q - structure of preimage of NQL}
 If $L$ is a subgroup of $Q$ then $N_G\pi^{-1}(L) = \pi^{-1} N_Q L$.
\end{Lemma}
\begin{proof}
If $g \in N_G\pi^{-1}(L)$ then $g^{-1}\pi^{-1}(L)g = \pi^{-1}(L)$ so applying $\pi$ gives that $\pi(g)^{-1}L\pi(g) = L$ and thus $g \in \pi^{-1}N_QL$.  

Conversely if $g \in \pi^{-1} (N_Q L)$ then $\pi(g)^{-1}L\pi(g) = L$ so 
\[\left( \pi^{-1}\circ\pi(g) \right)^{-1} \pi^{-1}(L) \left( \pi^{-1}\circ\pi(g) \right) = \pi^{-1}(L)\]
Since $g \in \pi^{-1}\circ\pi(g)$, $g^{-1}\pi^{-1}(L)g = \pi^{-1}(L)$.
\end{proof}

\begin{Prop}\label{prop:finite-by-uDn is uDn}
If $Q$ is Bredon duality of dimension $n$ over $R$ then $G$ is Bredon duality of dimension $n$ over $R$.
\end{Prop}

\begin{proof}
Let $K$ be a finite subgroup of $G$. We combine Lemma \ref{lemma:pi G to Q with finite kernel then NGK is finite index in etc} and Lemma \ref{lemma:pi G to Q and L 
subgroup of Q - structure of preimage of NQL} to see that $N_GK$ is finite index in $N_G (\pi^{-1} \circ \pi(K)) = \pi^{-1}\left( N_Q \pi(K) \right)$.  Hence 
\begin{align*}
H^i\left(W_GK , R[W_GK]\right) 
& \cong H^i\left(N_GK , R[N_GK]\right) \\
&\cong H^i\left(\pi^{-1} \left( N_Q \pi(K) \right), R\left[\pi^{-1}\left( N_Q\pi(K)\right) \right]\right) \\
&\cong H^i\left( N_Q \pi(K) , R\left[ N_Q\pi(K) \right]\right) \\
&\cong H^i\left( W_Q \pi(K) , R\left[ W_Q\pi(K) \right]\right)
\end{align*}
Where the third isomorphism follows from Lemma \ref{lemma:Finite by Q implies HG is HQ} and the short exact sequence
\[1 \longrightarrow F \longrightarrow \pi^{-1}\left( N_Q\pi(K)\right) \longrightarrow N_Q\pi(K) \longrightarrow 1 \]

Since $Q$ is Bredon duality of dimension $n$ this gives the condition on the cohomology of the Weyl groups.

$G$ is $\uFP_0$ by \ref{lemma:Finite-by-uFP0 is uFP0}, and $\ucd G = \ucd Q = n$ by \cite[5.5]{Nucinkis-OnDimensionsInBredonCohomology}.  So by Lemma \ref{lemma:OFFPn characterisation}, it remains to show that the Weyl groups of the finite subgroups are $\FP_\infty$.  For any finite subgroup $K$ of $G$, the short exact sequence above and \cite[Proposition 1.4]{Bieri-HomDimOfDiscreteGroups} gives that $\pi^{-1}\left( N_Q\pi(K)\right)$ is $\FP_\infty$.  But, as discussed at the beginning of the proof, $N_GK$ is finite index in $N_G (\pi^{-1} \circ \pi(K)) = \pi^{-1}\left( N_Q \pi(K) \right)$, so $N_G K$ is $\FP_\infty$ also. 
\end{proof}

Examining the proof above it's clear that $\mathcal{V}(G) = \mathcal{V}(Q)$.

\section{Graphs of Groups}
An amalgamated free product of two duality groups of dimension $n$ over a duality group of dimension $n-1$ is duality of dimension $n$, similarly an HNN extension of a duality group of dimension $n$ relative to a duality groups of dimension $n-1$ group is duality of dimension $n$ \cite[9.15]{Bieri-HomDimOfDiscreteGroups}.  We cannot hope for such a nice result as the normalisers of finite subgroups may be badly behaved, however there are some more restrictive cases where we can get results.  For instance using graphs of groups of Bredon duality groups we will be able to build Bredon duality groups $G$ with arbitrary $\mathcal{V}(G)$.  

We need some preliminary results, showing that a graph of groups is $\uFP$ if all groups involved are $\uFP$.  The following Proposition is well known over $\ZZ$, see for example \cite[Lemma 3.2]{Gandini-SomeH1FGroupsAndAConjectureOfKrophollerAndMislin}, and the proof extends with no alterations to arbitrary rings $R$.  See \cite{Serre} for the necessary background on Bass-Serre trees and graphs of groups.

\begin{Lemma}\label{lemma: duality mayervietoris of graph of groups}
There is an exact sequence, arising from the Bass-Serre tree.
 \[ \cdots \longrightarrow H^i_{\mf{F}} \left(G, -\right) \longrightarrow \bigoplus_{v \in V} H^i_{\mf{F}} \left( G_v, \Res^{G}_{G_v} - \right) \]
\[ \longrightarrow \bigoplus_{e \in E}H^i_{\mf{F}} \left(G_e, \Res^{G}_{G_e} - \right) \longrightarrow \cdots  \]
\end{Lemma}

\begin{Lemma}\label{lemma:uFPn for graph of groups}
 If all vertex groups $G_v$ are of type $\uFP_n$ and all edge groups $G_e$ are of type $\uFP_{n-1}$ over $R$ then $G$ is of type $\uFP_n$ over $R$.
\end{Lemma}

\begin{proof}
 Let $M_\lambda$, for $\lambda \in \Lambda$, be a directed system of Bredon-modules with colimit zero, for any subgroup $H$ of $G$ the directed system $\Res^{G}_{H} M_\lambda$ also has colimit zero.  The long exact sequence of Lemma \ref{lemma: duality mayervietoris of graph of groups}, and the exactness of colimits gives that for all $i$, there is an exact sequence 
\[ \cdots  \longrightarrow \varinjlim_{\lambda \in \Lambda} H^{i-1}_{\mf{F}} \left(G, M_\lambda \right) \longrightarrow \bigoplus_{v \in V} \varinjlim_{\lambda \in \Lambda} H^i_{\mf{F}} \left( G_v, \Res^{G}_{G_v} M_\lambda \right) \]
\[ \longrightarrow  \bigoplus_{e \in E} \varinjlim_{\lambda \in \Lambda} H^i_{\mf{F}} \left(G_e, \Res^{G}_{G_e} M_\lambda \right) \longrightarrow \cdots  \]
If $i \le n$ then by the Bieri-Eckmann criterion \cite[\S 5]{MartinezNucinkis-GeneralizedThompsonGroups}, the left and right hand terms vanish, thus the central term vanishes.  Another application of the Bieri-Eckmann criterion gives that $G$ is $\uFP_n$.
\end{proof}

\begin{Lemma}\label{lemma:ucd for graph of groups}
 If $\ucd_R G_v \le n$ for all vertex groups $G_v$ and $ \ucd_R G_e \le n-1$ for all edge groups $G_e$ then $\ucd_R G \le n$.
\end{Lemma}
\begin{proof}
Use the long exact sequence of Lemma \ref{lemma: duality mayervietoris of graph of groups}.
\end{proof}

\begin{Lemma}\label{lemma:cohomology for graph of groups}
 If there is some integer $n$ such that for all vertex groups $G_v$ and all edge groups $G_e$, $H^i(G_v, RG_v)$ is $R$-flat if $i = n$ and $0$ otherwise and $H^i(G_e, RG_e)$ is $R$-flat if $i = n-1$ and $0$ else, then $H^i(G, RG)$ is $R$-flat if $i = n$ and $0$ else.
\end{Lemma}
\begin{proof}
The Mayer-Vietoris sequence associated to the graph of groups is
\[ \cdots \longrightarrow H^q(G, RG) \longrightarrow \bigoplus_{v \in V} H^q\left(G_v, RG\right) \longrightarrow \bigoplus_{e \in E} H^q\left(G_e, RG \right) \longrightarrow \cdots \]
$H^q\left(G_v, RG\right) = H^q(G_v, RG_v) \otimes_{RG_v} RG$ by \cite[Proposition 5.4]{Bieri-HomDimOfDiscreteGroups} so we have 
\[ H^q(G, RG) = 0 \text{ for } q \neq n \]
and a short exact sequence
\begin{equation*}0 \longrightarrow \bigoplus_{e \in E} H^{n-1}(G_e, RG_e) \otimes_{RG_e} RG \longrightarrow H^n(G, RG) \tag{$\star$}\end{equation*}
\[ \longrightarrow \bigoplus_{v \in V} H^n(G_v, RG_v)\otimes_{RG_v} RG \longrightarrow 0\]
Finally, extensions of flat modules by flat modules are flat (use, for example, the long exact sequence associated to $\Tor^{RG}_*$).
\end{proof}

\begin{Remark}\label{remark:duality lemma:cohomology for graph of groups}
 In the lemma above, if $H^n(G, RG_v) \cong R$ and $H^{n-1}(G_e, RG_e) \cong R$ for all vertex and edge groups then $H^n(G, RG) $ will not be isomorphic to $R$.
\end{Remark}

\begin{Lemma}\label{lemma:duality normaliser in tree of groups}
Let $G$ be the fundamental group of a graph of groups $Y$.  If $K$ is a subgroup of the vertex group $G_v$ and $K$ is not subconjugate to any edge group then $N_G K = N_{G_v} K$. 
\end{Lemma}
\begin{proof}
The normaliser $N_G K$ acts on the $K$-fixed points of the Bass-Serre tree of $(G,Y)$, but only a single vertex is fixed by $K$, so necessarily $N_G K \le G_v$.
\end{proof}

\begin{Example}
Let $S_n$ denote the star graph of $n+1$ vertices---a single central vertex $v_0$, and a single edge connecting every other vertex $v_i$ to the central vertex.  Let $G$ be the fundamental group of the graph of groups on $S_n$, where the central vertex group $G_0$ is torsion-free duality of dimension $n$, the edge groups are torsion-free duality of dimension $n-1$ and the remaining vertex groups $G_i$ are Bredon duality of dimension $n$ with $H^n(G, RG) \neq 0$.

By Lemmas \ref{lemma:uFPn for graph of groups} and \ref{lemma:ucd for graph of groups}, $G$ is $\uFP$ of dimension $n$, so to prove it is Bredon duality it suffices to check the cohomology of the Weyl groups of the finite subgroups.  Any non-trivial finite subgroup is subconjugate to a unique vertex group $G_i$, and cannot be subconjugate to an edge group since they are assumed torsion-free.  If $K$ is a subgroup of $G_i$ then by Lemma \ref{lemma:duality normaliser in tree of groups}, $H^i(N_G K, R[N_GK]) \cong H^i(N_{G_i} K, R[N_{G_i}K])$ and the condition follows as $G_i$ was assumed to be Bredon duality.  Finally, for the trivial subgroup we must calculate $H^i(G, RG)$, which is Lemma \ref{lemma:cohomology for graph of groups}.

$\mathcal{V}(G)$ is easily calculable too, 
\[  \mathcal{V}(G) = \mathcal{V}(G_1) \vee \cdots \vee \mathcal{V}(G_n) \]
Where $\vee$ denotes the binary ``or'' operation.
\end{Example}

Specialising the above example:

\begin{Example}[A Bredon duality group with prescribed $\mathcal{V}(G)$]\label{example:duality arbitrary V} We specialise the above example.
 Let $\mathcal{V} = \{v_1, \ldots, v_t\} \subset \{0, 1, \ldots, n-1\}$ be given.  Choosing $G_i = \ZZ^n \rtimes \ZZ_2$ as in Example \ref{example:duality Zn antipodal} so that $\mathcal{V}(G_i) = v_i$, let $G_0 = \ZZ^n$, let the edge groups be $\ZZ^{n-1}$, and choose injections $\ZZ^{n-1} \to \ZZ^n$ and $\ZZ^{n-1} \to \ZZ^n \rtimes \ZZ_2$ from the edge groups into the vertex groups.  Then form the graph of groups as in the previous example to get, for $G$ the fundamental group of the graph of groups,
\[ \mathcal{V}(G) = \{v_1, \ldots, v_t\} \]
\end{Example}

 Because of Remark \ref{remark:duality lemma:cohomology for graph of groups} the groups constructed in the example above will not be Bredon--Poincar\'e duality groups.

\section{The Wrong Notion of Duality}\label{section:wrong notion of duality}

This section grew out of an investigation into which groups were $\uFP$ over some ring $R$ with
\[ H^i_{\mathfrak{F}}(G, R[ -, ? ]) \cong \left\{ \begin{array}{l l} \uR & \text{if $i = n$} \\ 0 & \text{else.} \end{array} \right. \]
One might hope that this na\"ive definition would give a duality similar to Poincar\'e duality, we show this is not the case.  Namely we prove in Theorem \ref{theorem:bredon wrong duality} that the only groups satisfying this property are torsion-free, and hence torsion-free Poincar\'e duality groups over $R$.  We need a couple of technical results before we can prove the theorem.

For $M$ a Bredon module, denote by $M^D$ the dual module 
\[M^D = \Mor_{\mathfrak{F}} \left( M(-), R[ -, ?] \right) \]
Note that $M^D$ is a covariant Bredon module.  Similarly for $A$ a covariant Bredon module:
\[A^D = \Mor_{\mathfrak{F}} \left( A(-), R[?, -] \right) \]

\begin{Lemma}\label{lemma:dual of constant functor}
 If $G$ is an infinite group and $\uR$ is the covariant constant functor on $R$ then $\uR^D = 0$.
\end{Lemma}
\begin{proof}
Observe that $\uR = R \otimes_{RG} R[G/1,-]$, then
\begin{align*}
\uR^D &= \Mor_{\mathfrak{F}}( \uR(?), R[-, ?] ) \\
 &\cong \Mor_{\mathfrak{F}}( R \otimes_{RG} R[G/1,?], R[-,?] ) \\
 &\cong \Hom_{RG}( R, \Mor_{\mathfrak{F}} ( R[G/1,?], R[-,?] ) ) \\
 &\cong \Hom_{RG}( R, R[-,G/1] )
\end{align*}
Where the second isomorphism is the adjoint isomorphism \cite[9.21]{Lueck} and the third is Lemma \ref{lemma:yoneda-type} (the Yoneda-type Lemma).  Finally, $\Hom_{RG}( R, R[-,G/1] )$ is the zero module since $G$ is infinite.
\end{proof}

\begin{Lemma}\label{lemma:double dual is nat iso}
 The dual functor takes projectives to projectives, and the double-dual functor $-^{DD}: \{\text{Bredon modules}\} \to \{\text{Bredon modules}\}$ is a natural isomorphism when restricted to the subcategory of finitely generated projectives.
\end{Lemma}
\begin{proof}
By the Yoneda-type Lemma \ref{lemma:yoneda-type},
\begin{align*}
        R[-,G/H]^D \cong \Mor_{\mathfrak{F}} (R[?,G/H], R[?, -]) \cong R[G/H, -]
       \end{align*}
For any module $M$, there is a natural map $\zeta: M \longrightarrow M^{DD}$, given by $\zeta (m)(f) = f(m)$.  If $M = R[-,G/H]$ then applying the Yoneda-type lemma twice shows $M^{DD} = M$.  The duality functor represents direct sums, showing the double dual of a projective is also a projective.
\end{proof}

\begin{Lemma}\label{lemma:R G/HG/K as an RWK module}
 There is an isomorphism of right $R[WK]$-modules
 \[R[G/H, -](G/K) = R[G/H, G/K] \cong \bigoplus_{\substack{g N_GK \in G/N_GK \\ g^{-1}Hg \le K}} R[WK]\]
\end{Lemma}
\begin{proof}
 Firstly, $R[G/H, G/K] \cong R[(G/K)^H]$ is a free $WK$-module, since if $n \in N_GK$ such that $gnK = gK$ then $nK = K$ and hence $n \in K$.  Now, $gK$ and $g^\prime K$ lie in the same $WK$ orbit if and only if $g(WK) K = g^\prime (WK) K$, equivalently $gN_GK = g^\prime N_GK$, and $gK$ determines an element of $R[(G/K)^K]$ if and only if $g^{-1}Hg \le K$.  Thus there is one $R[WK]$ orbit for each element in the set 
\[ \{ gN_GK \in G / N_G K \: : \: g^{-1} H g \le K\}  \]
\end{proof}

\begin{Lemma}\label{lemma:length n cov res then cdR<=n}
 If there exists a length $n$ resolution of the constant covariant module $\uR$ by projective covariant Bredon modules then $G$ is $R$-torsion free and $\cd_R G \le n$.
\end{Lemma}

\begin{proof}
Lemma \ref{lemma:R G/HG/K as an RWK module} above implies that evaluating a free covariant Bredon module at $G/1$ yields a free $RG$-module.  Thus evaluating a projective covariant Bredon module at $G/1$ yields a projective $RG$-module also.  

Let $P_* \longtwoheadrightarrow \uR$ be a length $n$ projective covariant resolution of $\uR$, evaluating at $G/1$ gives a length $n$ resolution of $R$ by projective $RG$-modules.  Thus $\cd_R G \le n$ and it follows that $G$ is $R$-torsion free.
\end{proof}

If $M$ is an $RG$-module then we denote by $IM$ the induced covariant Bredon module defined by $IM = M \otimes_{RG} R[G/1, -]$, or more explicitly:
 \[ IM : G/H \mapsto M \otimes_{RG} R[G/1, G/H] \]
The functor $I$ maps projective modules to projective modules and, by a proof analagous to \cite[9.21]{Lueck}, satisfies the following adjoint isomorphism for any covariant Bredon module $A$
\[\Mor_{\mathfrak{F}}( IM, A ) \cong \Hom_{RG}(M, A(G/1) )\]

\begin{Lemma}\label{lemma:cdR<=n then length n cov res}
 If $\cd_R G \le n$ then there exists a length $n$ projective covariant resolution of $\uR$.
\end{Lemma}
\begin{proof}
 Let $P_*$ be a length $n$ projective $RG$-module resolution of $R$, then we claim $IP_*$ is a projective covariant resolution of $\uR$.  One can easily check that $IR = \uR$ so it remains to show $IP_*$ is exact.  Evaluating at $G/H$ gives
\begin{align*}
 IP_*(G/H) ) &\cong P_* \otimes_{RG}R[G/H] \\
 &\cong P_* \otimes_{RH} R
\end{align*}
Since $\cd_R G < \infty$, $G$ is $R$-torsion free, thus $\lvert H \rvert$ is invertible in $R$ and $R$ is projective over $RH$ \cite[Proposition 4.12(a)]{Bieri-HomDimOfDiscreteGroups}.
\end{proof}

\begin{Theorem}\label{theorem:bredon wrong duality}
If $G$ is an arbitrary group $\uFP$ group with $\ucd_R G = n$ and 
\[H^i_{\mathfrak{F}}(G, R[ -, ? ]) \cong \left\{ \begin{array}{l l} \uR & \text{if $i = n$} \\ 0 & \text{else.} \end{array} \right. \]
then $G$ is torsion-free.  Note that in the above, $\uR$ denotes the constant covariant Bredon module.
\end{Theorem}
\begin{proof}
Choose a length $n$ finite type projective Bredon module resolution $ P_*$ of $\uR$ then by the assumption on $H^n_{\mathfrak{F}}(G, R[ -, ? ])$, $P_*^D$ is a covariant resolution by finitely generated projectives of $\uR$:
\[0 \longrightarrow  P_0^D(-) \stackrel{\partial_1^D}{\longrightarrow} P_1^D(-) \stackrel{\partial_2^D}{\longrightarrow} \cdots \stackrel{\partial_n^D}{\longrightarrow} P_n^D(-) \longrightarrow H^n_{\mathfrak{F}}(G, R[  -, ? ]) \cong \uR(?)  \longrightarrow 0\]
By Lemma \ref{lemma:length n cov res then cdR<=n} $G$ is $R$-torsion-free and $\cd_R G \le n$.  Since $G$ is $\uFP_\infty$, $G$ is $\FP_\infty$ (Lemma \ref{lemma:OFFPn characterisation}) and we may choose a length $n$ finite type projective $RG$-resolution $Q_*$ of $R$.  Lemma \ref{lemma:cdR<=n then length n cov res} gives that $IQ_* \longtwoheadrightarrow \uR$ is a projective covariant resolution.  

By the Bredon analog of the comparison theorem \cite[2.2.6]{Weibel}, the two projective covariant resolutions of $\uR$ are chain homotopy equivalent.  Any additive functor preserves chain homotopy equivalences, so applying the dual functor to both complexes gives a chain homotopy equivalence between
\[0 \longrightarrow \uR^D \cong 0 \longrightarrow (I Q_0)^D \longrightarrow  \cdots \longrightarrow (I Q_n)^D \]
and
\[0 \longrightarrow \uR^D \cong 0 \longrightarrow P_n^{DD} \longrightarrow  P_{n-1}^{DD} \longrightarrow \cdots \longrightarrow P_0^{DD} \]
(That $\uR^D \cong 0 $ is Lemma \ref{lemma:dual of constant functor}).  Since $\Mor_{\mathfrak{F}}$ is left exact we know both complexes above are left exact.  Lemma \ref{lemma:double dual is nat iso} gives the commutative diagram below.
\[
\xymatrix{
 0 \ar[r] & P_n^{DD} \ar[r] \ar^\cong[d] &  \cdots \ar[r] & P_1^{DD} \ar[r] \ar^\cong[d] &  P_0^{DD}  \ar^\cong[d]\\
 0 \ar[r] & P_n \ar[r] &  \cdots \ar[r]  & P_1 \ar[r] & P_0
}\]

The lower complex, $P_*$, satisfies $H_0P_* \cong \uR$ and $H_iP_* = 0$ for all $i \neq 0$.  Thus the same is true for the top complex, and also the complex $I Q_*^D$, since this is homotopy equivalent to it.  In particular, there is an epimorphism of Bredon modules, 
\[I Q_n^D \longtwoheadrightarrow \uR\]
The left hand side simplifies, using the adjoint isomorphism
\[I Q_n^D = \Mor_{\mathfrak{F}} \left( I Q_n , R[?, -] \right) \cong \Hom_{RG}(Q_n, R[?, G/1]) \]
Since $\Hom_{RG}(Q_n, R[?, G/1])(G/H) = 0$ if $H \neq 1$, this module cannot surject onto $\uR$ unless $G$ is torsion-free. 
\end{proof}

\providecommand{\bysame}{\leavevmode\hbox to3em{\hrulefill}\thinspace}
\providecommand{\MR}{\relax\ifhmode\unskip\space\fi MR }
\providecommand{\MRhref}[2]{%
  \href{http://www.ams.org/mathscinet-getitem?mr=#1}{#2}
}
\providecommand{\href}[2]{#2}

\end{document}